\documentclass[11pt,a4paper]{article}

\usepackage[utf8]{inputenc}
\usepackage[T1]{fontenc}
\usepackage[english]{babel} 
\usepackage{amsmath}
\usepackage{amsfonts}
\usepackage{amsthm}
\usepackage{amssymb}
\usepackage{makeidx}
\usepackage{graphicx}
\usepackage{lmodern}
\usepackage[left=2cm,right=2cm,top=2cm,bottom=2cm]{geometry}

\usepackage{color}
\usepackage{comment}
\usepackage[dvipsnames]{xcolor}
\usepackage{graphicx}
\usepackage{framed}
\usepackage{tikz}
\usepackage{bm}

\usepackage{fourier-orns}
\usepackage{mathtools}
\usepackage{relsize}
\usepackage{dsfont}
\usepackage{comment}

\usepackage{hyperref}
\hypersetup{
	linktocpage = true,
	colorlinks = true,
	linkcolor = {RoyalBlue},
	urlcolor = {RoyalBlue},
	citecolor = {Green}}

\newtheorem{Def}{Definition}[section]
\newtheorem{thm}[Def]{Theorem}
\newtheorem{prop}[Def]{Proposition}
\newtheorem{rmk}[Def]{Remark}
\newtheorem{lem}[Def]{Lemma}

\newtheorem{taggedhypxmas}{Main Assumptions on the Stochastic Dynamics --}
\newenvironment{taggedhypmas}[1]
    {\renewcommand\thetaggedhypxmas{#1}\taggedhypxmas}
    {\endtaggedhypxmas}

\newtheorem{taggedhypxmac}{Main Assumptions on the Cost and Constraints --}
\newenvironment{taggedhypmac}[1]
    {\renewcommand\thetaggedhypxmac{#1}\taggedhypxmac}
    {\endtaggedhypxmac}
    
\newtheorem{taggedhypxacd}{Additional Assumptions for Controlled Diffusion --}
\newenvironment{taggedhypacd}[1]
    {\renewcommand\thetaggedhypxacd{#1}\taggedhypxacd}
    {\endtaggedhypxacd}

\title{First-Order Pontryagin Maximum Principle for Risk-Averse Stochastic Optimal Control Problems}

\author{R. Bonalli\thanks{Laboratoire des Signaux et Syst\`emes, Universit\'e Paris-Saclay, CNRS, CentraleSup\'elec \textit{E-mail}: \texttt{riccardo.bonalli@l2s.centralesupelec.fr}} \,and Benoît Bonnet\thanks{LAAS-CNRS, Université de Toulouse, CNRS, 7 avenue du colonel Roche, F-31400 Toulouse, France. \textit{E-mail}: \texttt{benoit.bonnet@laas.fr}}}

\begin{document}

\maketitle

\begin{abstract}
In this paper, we derive first-order Pontryagin optimality conditions for risk-averse stochastic optimal control problems subject to final time inequality constraints, and whose costs are general, possibly non-smooth finite coherent risk measures. Unlike preexisting contributions covering this situation, our analysis holds for classical stochastic differential equations driven by standard Brownian motions. In addition, it presents the advantages of neither involving second-order adjoint equations, nor leading to the so-called weak version of the PMP, in which the maximization condition with respect to the control variable is replaced by the stationarity of the Hamiltonian.
\end{abstract}

{\footnotesize

\vspace{0.25cm}

\textbf{MSC2020 Subject Classification :}   93E03, 93E20.
}

\tableofcontents


\section{Introduction} 
\setcounter{equation}{0} \renewcommand{\theequation}{\thesection.\arabic{equation}}

In the last decades, risk-averse stochastic optimal control has seen a surge of interest as a tool for designing control laws that enjoy robustness properties against uncertainties. Relevant applications of this theory encompass broad research fields, ranging from risk-averse financial investments to the safe control of autonomous systems, as evidenced e.g. by the recent monographs \cite{Chapman2021,Shapiro2021} and their bibliography. In this context, first-order necessary conditions for optimality in the form of Pontryagin's Maximum Principle (we will refer to these latter as ``risk-averse PMP'' in the sequel) are bound to play a key role in characterizing and numerically computing optimal control strategies, as it is known to be the case for classical stochastic optimal control problems in which only expectation-based costs and constraints are considered \cite{Yong1999}. However, extending the PMP in its general form to more involved risk-averse settings still requires substantial investigations.

To the best of our knowledge, the derivation of a risk-averse PMP was attempted firstly in \cite{Sun2017}, where appropriate adjoint equations and maximality conditions formulated in terms of the so-called $G$-Stochastic calculus are introduced in order to cope with the presence of risk measures. This framework was originally introduced by Peng \cite{Peng2007}, and developed by the stochastic control community later on, see e.g. \cite{Redjil2017,Soner2011}. In this setting, the standard Brownian motion is replaced by a so-called $G$-Brownian motion, which is modelled as a stochastic process whose distribution is the product of a standard Gaussian and a Lipschitz map, and whose role is to transform the coherent risk measure into a standard, though non-linear expectation. While practical for some applications, this procedure requires to change the dynamics of the system, which is not always natural e.g. when the diffusion term aims at rendering an unknown uncertainty exerted on the system by the environment. Therefore, for certain classes of problems, it is still relevant to investigate optimality conditions relying on standard stochastic calculus, and which do not require to infuse additional uncertainty in the formulation of the control problem. Along this line, a risk-averse PMP for problems which are subject to stochastic differential equations stemming from classical Wiener processes is proposed in \cite{Isohatala2021}, though no final constraints are included therein and the underlying risk measures are assumed to be continuously Fr\'echet differentiable. From a different standpoint, first-order necessary optimality conditions for convex risk-averse optimization problems subject to partial differential equations and general subdifferentiable risk-measure-based costs are derived in \cite{Garreis2021}, by leveraging classical tools from convex analysis. Nevertheless, final constraints are also ruled out in this work, and the necessary conditions for optimality are written down as simple Euler conditions and not as a general Karush-Kuhn-Tucker system, which would be the natural ``static'' counterpart of the PMP.

In this paper, we propose a first step towards bridging the aforedescribed gap by establishing a first-order risk-averse PMP for a class of finite-dimensional constrained stochastic optimal control problems. Therein, one aims at minimizing a final cost modelled as a general subdifferentiable coherent risk measure over a class of admissible trajectories driven by a controlled stochastic differential equations involving standard Wiener processes, and subject to final time inequality constraints. Our proof leverages a general methodology that was first developed in \cite{Frankowska1987}, allowing for a natural extension of the first-order PMP for stochastic optimal control problems with expectation-based costs discussed in \cite{Frankowska2020} to the risk-averse setting. Specifically, the main advantages offered by this approach over more classical needle-like variations or Ekeland's principle-based methods are twofold. Firstly, no additional second-order adjoint variables (nor related second-order adjoint equations) are required to establish a fully informative PMP. Secondly, it permits the derivation of the so-called strong maximum principle, in which the optimal controls are characterized as being pointwise maximizers of the Hamiltonian. This is in contrast with some reference contributions in stochastic optimal control that establish weaker variants of the PMP in which the maximization condition is relaxed by requiring the stationarity of the Hamiltonian \cite{Frankowska2017,Frankowska2018}.  In what follows, we propose two separate sets of optimality conditions for the class of optimal control problems at hand, depending on whether the control variable appears in the diffusion term or not. When the control acts only on the deterministic drift, the variational linearization techniques subtending the proof of the maximum principle can be performed much like in the deterministic case, by considering perturbations which are tangent to the set of relaxed velocities. When the diffusion is controlled, however, it is not possible to replicate such a strategy as the It\^o integral does not exhibit the nice convexifying effects of the Lebesgue or Bochner integrals -- a fact which is expounded by an original example in Remark \ref{rmk:Ito} --, and one thus needs to impose an a priori convexity assumptions on the sets of admissible drift and diffusion pairs, similar to that considered e.g. in \cite{Frankowska2020}. 

The paper is organized as follows. In Section \ref{sec:Preli}, we recollect known concepts of stochastic calculus and set-valued analysis, which feature a counterexample to Aumann's theorem for the It\^o integral that we believe to be of independent interest. In Section \ref{sec:PMP}, we expose the main contributions of this article, which are first-order Pontryagin optimality conditions for risk-averse stochastic optimal control problems. We start in Section \ref{sec:controlledDiff} with the case in which the diffusion term of the driving stochastic dynamics is controlled, and expose the proof in great details in this context. We then show in Section \ref{sec:uncontrolledDiff} how the aforeproposed methodology can be used to prove the PMP under more general assumptions when the diffusion term is control-free, and close the paper with Sections \ref{sec:Example} and \ref{sec:conclusion} which respectively contain some application examples and important perspectives.


\section{Preliminaries}
\label{sec:Preli}

\setcounter{equation}{0} \renewcommand{\theequation}{\thesection.\arabic{equation}}

In this section, we recollect some useful concepts and results of stochastic calculus, for which we mainly refer to \cite{LeGall2016,Yong1999}, as well as notions of set-valued analysis mostly excerpted from \cite{Aubin1990}. From now on, we fix positive integers $n , m , d \in \mathbb{N}$, a finite time horizon $T > 0$, and let $\beta \in [1,+\infty)$.

\subsection{Stochastic Calculus}

Throughout this article, we will consider random variables defined over a probability space $(\Omega,\mathcal{G},\mathbb{P})$. For any sub $\sigma$-algebra $\mathcal{S} \subset \mathcal{G}$, we denote by $L^{\beta}_{\mathcal{S}}(\Omega,\mathbb{R}^n)$ the Banach space of random variables $z : \Omega \rightarrow \mathbb{R}^n$ which are $\mathcal{S}$-measurable and such that
\begin{equation*}
\| z \|_{L^{\beta}} \triangleq \mathbb{E} \big[ \, \| z \|^{\beta} \, \big]^{1/\beta} < \infty,
\end{equation*}
where $\| \cdot \|$ denotes the Euclidean norm. It is a standard consequence of Riesz's theorem that $L^{\beta}_{\mathcal{S}}(\Omega,\mathbb{R})^*$ is isomorphic to $L^{\gamma}_{\mathcal{S}}(\Omega,\mathbb{R})$ where $\gamma \in(1,+\infty]$ satisfies $1/\beta + 1/\gamma=1$. 

Let $(W_s)_{s \in [0,T]} = (W^1_s,\dots,W^d_s)_{s \in[0,T]}$ be a $d$--dimensional Wiener process which generates a complete filtration 
\begin{equation*}
 \mathcal{F} \triangleq (\mathcal{F}_t)_{t \in [0,T]} = \Big( \sigma \big( W_s : 0 \le s \le t \big) \Big)_{t \in [0,T]},
\end{equation*}
and denote by  $L^{\beta}_{\mathcal{F}}([0,T] \times \Omega,\mathbb{R}^n)$ the corresponding Banach space of $\mathcal{F}$-progressively measurable -- or progressively measurable -- processes $x :[0,T] \times\Omega \rightarrow \mathbb{R}^n$ which satisfy
\begin{equation*}
\| x \|_{L^{\beta}_{\mathcal{F}}} \triangleq \mathbb{E}\bigg[ \int^T_0 \| x(s) \|^{\beta} \; \mathrm{d}s \bigg]^{1/\beta} < \infty. 
\end{equation*}
In addition, denote by $C^{\beta}_{\mathcal{F}}([0,T] \times \Omega,\mathbb{R}^n)$ the Banach space of $\mathcal{F}$-adapted processes $x : [0,T] \times\Omega \rightarrow \mathbb{R}^n$ which have continuous sample paths and finite sup norm, namely
\begin{equation*}
\| x \|_{C^{\beta}_{\mathcal{F}}} \triangleq \mathbb{E} \bigg[ \, \underset{s \in [0,T]}{\sup} \ \| x(s) \|^{\beta} \bigg]^{1/\beta} < \infty.
\end{equation*}
In particular, $C^{\beta}_{\mathcal{F}}([0,T] \times \Omega,\mathbb{R}^n) \subset L^{\beta}_{\mathcal{F}}([0,T] \times \Omega,\mathbb{R}^n)$. In the sequel given  $t \in [0,T]$, we will often use the standard notation $x(t) : \Omega \to \mathbb{R}^n$ to refer to progressively measurable processes. In addition, when we say that a property holds ``almost everywhere'', it shall always be understood with respect to the progressive $\sigma$-algebra generated by the filtration $\mathcal{F}$ on $[0,T]\times\Omega$.

An $\mathcal{F}$-adapted process $x : [0,T] \times\Omega \rightarrow \mathbb{R}^n$ such that $x(s) \in L^1_{\mathcal{F}_s}(\Omega,\mathbb{R}^n)$ for every $s \in [0,T]$ is called a \textit{martingale} provided that
\begin{equation*}
\mathbb{E}\big[ x(t) | \mathcal{F}_s \big] = x(s),
\end{equation*}
for all $0 \leq s < t \leq T$. We then say that a martingale $x : [0,T] \times\Omega \rightarrow \mathbb{R}^n$ is uniformly bounded in $L^{\beta}$ if there exists a constant $C > 0$ such that
\begin{equation*}
\underset{t \in [0,T]}{\sup} \ \| x(t) \|_{L^{\beta}} \le C .
\end{equation*}
In this setting, for every $x \in L^{\beta}_{\mathcal{F}}([0,T] \times \Omega,\mathbb{R}^n)$ and each $i \in \{1,\dots,d\}$, we write
\begin{equation*}
y^i : t \in[0,T] \mapsto \int^{t}_0 x(s) \; \mathrm{d}W^i_s    
\end{equation*}
for the It\^o integral of $x$ with respect to $W^i$, and recall that $y^i$ is then a martingale which is additionally in $C^{\beta}_{\mathcal{F}}([0,T] \times \Omega,\mathbb{R}^n)$. Analogously, we introduce the notation
\begin{equation*}
y(t) = \int^t_0 x(s) \; \mathrm{d}W_s \triangleq \sum^d_{i=1} \int^t_0 x(s)^i \; \mathrm{d}W^i_s
\end{equation*}
for $x \in L^{\beta}_{\mathcal{F}}([0,T] \times \Omega,\mathbb{R}^{n \times d})$, where $x(s) = ( x(s)^1 | \dots | x(s)^d )$ and $x(s)^i \in L^{\beta}_{\mathcal{F}_s}(\Omega,\mathbb{R}^n)$, and recall the famed \textit{Burkholder-Davis-Gundy} inequality 
\begin{equation}
\mathbb{E} \bigg[ \sup_{t \in[0,T]} \|y(t)\|^{\beta} \bigg] \leq C_{\beta} \mathbb{E} \Bigg[ \bigg( \int_0^T \|x(t)\|^2 \mathrm{d}t \bigg)^{\hspace{-0.075cm} \beta/2} \, \Bigg]
\end{equation}
which holds for some constant $C_{\beta}>0$ that only depends on $\beta\in[1,+\infty)$. The following representation theorem for martingales (see e.g. \cite[Theorem 5.18]{LeGall2016}) will be crucial in the derivation of the adjoint dynamics of the PMP in Section \ref{sec:PMP}.
\begin{thm}[Martingale representation theorem]
\label{thm:Mart}
Let $x : [0,T] \times\Omega \rightarrow \mathbb{R}^n$ be a martingale which is uniformly bounded in $L^2$. Then, there exist a vector $N \in \mathbb{R}^n$ and a stochastic process $\mu \in L^2_{\mathcal{F}}([0,T] \times \Omega,\mathbb{R}^{n \times d})$ such that
$$
x(t) = N + \int^t_0 \mu(s) \, \mathrm{d}W_s \qquad \textnormal{for all} \ t \in [0,T] .
$$
\end{thm}

In this article, we will study risk-averse stochastic optimal control problems, which involve the following class of functionals called \textit{finite coherent risk measures}, whose properties are extensively studied in \cite{Shapiro2021}.

\begin{Def}[Finite coherent risk measure]
\label{def:Risk}
A mapping $\rho : L^1_{\mathcal{S}}(\Omega,\mathbb{R}) \rightarrow \mathbb{R}$ is called a \textnormal{finite coherent risk measure} if it satisfies the following properties. 
\begin{enumerate}
    \item (Convexity) For every $Z_1 , Z_2 \in L^1_{\mathcal{S}}(\Omega,\mathbb{R})$ and all $\lambda \in [0,1]$, it holds 
    $$
    \rho(\lambda Z_1 + (1 - \lambda) Z_2) \le \lambda \rho(Z_1) + (1 - \lambda) \rho(Z_2).
    $$
    \item (Monotonicity) If $Z_1 , Z_2 \in L^1_{\mathcal{S}}(\Omega,\mathbb{R})$ are such that $Z_1 \le Z_2$, then 
    \begin{equation*}
    \rho(Z_1) \le \rho(Z_2).
    \end{equation*}    
    \item (Translation invariance) For every $Z \in L^1_{\mathcal{S}}(\Omega,\mathbb{R})$ and $\alpha \in \mathbb{R}$, it holds 
    \begin{equation*}
    \rho(Z + \alpha) = \rho(Z) + \alpha.    
    \end{equation*}
    \item (Positive homogeneity) For every $Z \in L^1_{\mathcal{S}}(\Omega,\mathbb{R})$ and $\alpha > 0$, it holds 
    \begin{equation*}
    \rho(\alpha Z) = \alpha \rho(Z).
    \end{equation*}
\end{enumerate}
\end{Def}

As detailed throughout \cite[Chapter 6]{Shapiro2021}, coherent risk measures satisfy the following fundamental properties.
\begin{thm}[Structure of finite coherent risk measures]
\label{thm:RiskMeasures}
Given a finite coherent risk measure $\rho : L^1_{\mathcal{S}}(\Omega,\mathbb{R}) \rightarrow \mathbb{R}$, the following holds true.
\begin{enumerate}
    \item For every $Z \in L^1_{\mathcal{S}}(\Omega,\mathbb{R})$, the risk measure can be represented as 
    $$
    \rho(Z) = \underset{\xi \in \partial\rho(0)}{\sup} \ \mathbb{E}[ \xi Z ] ,
    $$
    where $\partial\rho(0)$ denotes the convex subdifferential of $\rho$ at $Z=0$.
    \item For every $Z \in L^1_{\mathcal{S}}(\Omega,\mathbb{R})$, the subdifferential $\partial\rho(Z) \subset L^{\infty}_{\mathcal{S}}(\Omega,\mathbb{R})$ is a nonempty, convex, and weakly-$^*$ compact set which can be expressed as
    $$
    \partial\rho(Z) = \underset{\xi \in \partial \rho(0)}{\arg \max} \ \mathbb{E}[ \xi Z ] .
    $$
    \item For every $Z , H \in L^1_{\mathcal{S}}(\Omega,\mathbb{R})$, the mapping $\rho$ has a sublinear directional derivative $D\rho(Z) \cdot H$ at $Z$ along $H$, which satisfies
    $$
    D\rho(Z) \cdot H = \underset{\xi \in \partial\rho(Z)}{\max} \ \mathbb{E}[ \xi H ] .
    $$
\end{enumerate}
\end{thm}

As previously mentioned in the introduction, coherent risk measures appear very naturally in a broad range of stochastic decision problems, with their most common representative being the \textit{Average Value-at-Risk}, see e.g. \cite[Section 6.2.4]{Shapiro2021} and the examples of Section \ref{sec:Example} below.


\subsection{Stochastic Differential Equations}

In what follows, we detail the setting in which we study controlled stochastic dynamics. Let $U \subset \mathbb{R}^m$ be a compact set representing admissible control values, and consider a stochastic drift mapping $f : [0,T] \times \Omega \times \mathbb{R}^n \times U \rightarrow \mathbb{R}^n$ as well as a stochastic diffusion mapping $\sigma : [0,T] \times \Omega \times \mathbb{R}^n \times U \rightarrow \mathbb{R}^{n \times d}$ which satisfy the following series of standard assumptions (see e.g. \cite[Chapter 3.3]{Yong1999}).

\vspace{10pt}

\begin{flushleft}
\begin{taggedhypmas}{(MSD)} \hfill
\label{hyp:MAS}
\begin{enumerate}
    \item[$(i)$] The applications
    \begin{equation*}
    f(\cdot,\cdot,x,u) : [0,T]\times\Omega \rightarrow \mathbb{R}^n , \quad \sigma(\cdot,\cdot,x,u) : [0,T]\times\Omega \rightarrow \mathbb{R}^{n \times d} ,
    \end{equation*}
    are progressively measurable for every $(x,u) \in \mathbb{R}^n \times U$ and the maps
    \begin{equation*}
    f(t,\omega,\cdot,\cdot) : \mathbb{R}^n \times U \rightarrow \mathbb{R}^n , \quad \sigma(t,\omega,\cdot,\cdot) : \mathbb{R}^n \times U \rightarrow \mathbb{R}^{n \times d}
    \end{equation*}
    are continuous for almost every $(t,\omega) \in [0,T]\times\Omega$.
    \item[$(ii)$] There exists a map $k \in L^2_{\mathcal{F}}([0,T]\times\Omega,\mathbb{R}_+)$ such that\,\footnote{Note that since $U\subset\mathbb{R}^m$ is compact, this assumption encompasses control-affine dynamics.}
    \begin{equation*}
    \| f(t,\omega,0,u) \| + \| \sigma(t,\omega,0,u) \| \le k(t,\omega),
    \end{equation*}
    for almost every $(t,\omega) \in [0,T] \times \Omega$ and each $u \in U$.
    \item[$(iii)$] For almost every $(t,\omega) \in [0,T] \times \Omega$ and all $u\in U$, the mappings
    \begin{equation*}
    f(t,\omega,\cdot,u) : \mathbb{R}^n \rightarrow \mathbb{R}^n , \quad \sigma(t,\omega,\cdot,u) : \mathbb{R}^n \rightarrow \mathbb{R}^{n \times d} ,
    \end{equation*}
    are Fr\'echet differentiable, and there exists a constant $L > 0$ such that
    $$
    \left\| \frac{\partial f}{\partial x}(t,\omega,x,u) \right\| + \left\| \frac{\partial \sigma}{\partial x}(t,\omega,x,u) \right\| \le L ,
    $$
    and
    $$
    \left\| \frac{\partial f}{\partial x}(t,\omega,x,u) - \frac{\partial f}{\partial x}(t,\omega,y,u) \right\| + \left\| \frac{\partial \sigma}{\partial x}(t,\omega,x,u) - \frac{\partial \sigma}{\partial x}(t,\omega,y,u) \right\| \le L \| x - y \| ,
    $$
    for almost every $(t,\omega) \in [0,T] \times \Omega$, any $u\in U$ and all $x , y \in \mathbb{R}^n$.
\end{enumerate}
\end{taggedhypmas}
\end{flushleft}

\vspace{10pt}

From now on, we fix an initial condition $x_0 \in L^2_{\mathcal{F}_0}(\Omega,\mathbb{R}^n)$. Under hypotheses \ref{hyp:MAS}
, the stochastic differential equation
\begin{equation}
\label{eq:SDE}
\tag{$\textnormal{SDE}$}
\left\{
\begin{aligned}
\mathrm{d}x(t) &  = f(t,x(t),u(t)) \mathrm{d}t + \sigma(t,x(t),u(t))  \mathrm{d}W_t , \\
x(0) & = x_0,
\end{aligned}
\right.
\end{equation}
has a unique (up to stochastic indistinguishability) solution $x_u \in C^2_{\mathcal{F}}([0,T]\times\Omega,\mathbb{R}^n)$ for every progressively measurable control $u : [0,T] \times \Omega \rightarrow U$. In the following lemma, we recall a useful estimate for this class of dynamics (see e.g. \cite[Proposition 2.1]{Libin2007}).

\begin{lem} 
\label{lemma:bound}
Let $u : [0,T] \times \Omega \rightarrow U$ be a progressively measurable control signal and suppose that assumptions \textnormal{\ref{hyp:MAS}}. 
Then, the corresponding solution $x_u \in C^2_{\mathcal{F}}([0,T]\times\Omega,\mathbb{R}^n)$ of \eqref{eq:SDE} satisfies the estimate
$$
\| x_u \|_{C^1_{\mathcal{F}}} \le C \hspace{0.05cm} \mathbb{E}\left[ \| x_0 \| + \int^T_0 \| f(s,0,u(s)) \| \, \mathrm{d}s + \left( \int^T_0 \| \sigma(s,0,u(s)) \|^2 \, \mathrm{d}s \right)^{\hspace{-0.1cm} 1/2} \, \right],
$$
where the constant $C > 0$ only depends on the magnitudes of $T$ and $L$.
\end{lem}


\subsection{Set-valued Analysis}

In the sequel given a closed set $K \subset \mathbb{R}^n$, we define  its \textit{closed convex hull} by
\begin{equation}
\label{eq:ConvHull}
\overline{\textnormal{co}} K := \overline{\bigcup_{N \geq1} \bigg\{ \sum_{i=1}^N \alpha_i x_i \, : \, x_i \in K, \, \alpha_i \geq 0 \,~\text{for $i \in \{1,\dots,N\}$ and}~ \sum_{i=1}^N \alpha_i =1 \bigg\}}.
\end{equation}
If the set $K$ is convex, we shall denote its \textit{tangent cone} at some $x \in K$ by
\begin{equation}
\label{eq:TangentCone}
T_K(x) := \overline{\Big\{ v \in \mathbb{R}^n : \lim_{h \to 0^+} \tfrac{1}{h} \textnormal{dist}_K(x+ hv) =0 \Big\}} = \overline{\bigcup_{\lambda > 0} \lambda (K-x)}, 
\end{equation}
where $\textnormal{dist}_K(x) := \inf_{y \in K} \| x-y \|$ denotes the distance from a point $x \in\mathbb{R}^n$ to $K$.

We will write $F : [0,T] \times\Omega \rightrightarrows \mathbb{R}^n$ to denote a \textit{set-valued map} -- or \textit{multifunction} -- from $[0,T]\times\Omega$ into $\mathbb{R}^n$, namely a mapping valued in the subsets of $\mathbb{R}^n$. In this context, we shall say that $F$ has closed, compact or convex images if its values are closed, compact or convex sets respectively. 

\begin{Def}[Progressively measurable set-valued maps]
We say that a set-valued map $F : [0,T] \times\Omega \rightrightarrows \mathbb{R}^n$ is \textnormal{progressively measurable} if
\begin{equation*}
F^{-1}(\mathcal{O}) := \Big\{ (t,\omega) \in [0,T] \times \Omega \, : \, F(t,\omega) \cap \mathcal{O} \neq \emptyset \Big\}
\end{equation*}
is measurable with respect to the progressive $\sigma$-algebra generated by the filtration $\mathcal{F}$ on $[0,T]\times\Omega$ for every open set $\mathcal{O} \subset \mathbb{R}^n$.
\end{Def}

We recall in the following theorem a direct consequence of \cite[Theorem 8.1.3]{Aubin1990}.

\begin{thm}[Existence of progressively measurable selections]
A progressively measurable set-valued map $F : [0,T]\times \Omega \rightrightarrows \mathbb{R}^n$ with nonempty closed images admits a \textnormal{progressively measurable selection}, namely a progressively measurable function $f : [0,T]\times \Omega \to \mathbb{R}^n$ such that $f(t,\omega) \in F(t,\omega)$ for almost every $(t,\omega) \in [0,T] \times \Omega$.
\end{thm}

In the following definitions, we recall classical adaptations of the concepts of integral boundedness and Lipschitz regularity for progressively measurable set-valued maps with compact images. The latter of these properties is expressed in terms of the so-called \textit{Pompeiu-Hausdorff} distance, defined by 
\begin{equation*}
d_{\mathcal{H}}(A,B) := \max\bigg\{ \sup_{x \in A} \textnormal{dist}_B(x) \, , \, \sup_{y \in B} \textnormal{dist}_A(y) \bigg\}
\end{equation*}
for any pair of compact sets $A,B \subset \mathbb{R}^n$. 

\begin{Def}[Integrably bounded multifunction]
A set-valued mapping $F : [0,T] \times \Omega \times \mathbb{R}^n \rightrightarrows \mathbb{R}^n$ with nonempty compact images is \textnormal{integrably bounded} if 
\begin{equation*}
F(t,\omega,x) \subset k(t,\omega) \mathbb{B}
\end{equation*}
for almost every $(t,\omega) \in[0,T] \times \Omega$ and all $x \in\mathbb{R}^n$, where $k \in L^2_{\mathcal{F}}([0,T]\times\Omega,\mathbb{R}_+)$ and $\mathbb{B} \subset\mathbb{R}^n$ denotes the closed unit ball centered at the origin.
\end{Def}

\begin{Def}[Progressively measurable-Lipschitz multifunction]
We say that a set-valued mapping $F : [0,T] \times \Omega \times \mathbb{R}^n \rightrightarrows \mathbb{R}^n$ with nonemtpy compact images is \textnormal{progressively measurable-Lipschitz} if
\begin{equation*}
(t,\omega) \in [0,T] \times \Omega \rightrightarrows F(t,\omega,x) \in\mathbb{R}^n, 
\end{equation*}
is progressively measurable for each $x \in \mathbb{R}^n$, and there exists $L>0$ such that
\begin{equation*}
d_{\mathcal{H}}(F(t,\omega,x),F(t,\omega,y)) \leq L |x-y|, 
\end{equation*}
for almost every $(t,\omega) \in [0,T] \times \Omega$ and all $x,y \in \mathbb{R}^n$.
\end{Def}

We recall in the following theorem some classical adaptations of \cite[Corollary 8.2.13, Theorem 8.5.1, Corollary 8.5.2]{Aubin1990}, which ensure the existence of progressively measurable selections for various classes of set-valued mappings. 

\begin{thm}[Some progressively measurable selection results]
\label{thm:Selections}
Let $F : [0,T] \times \Omega \times \mathbb{R}^n \rightrightarrows \mathbb{R}^n$ be progressively measurable-Lipschitz with nonempty compact images, fix $x,y \in C^{\beta}_{\mathcal{F}}([0,T] \times\Omega,\mathbb{R}^n)$ and $l \in  L^{\beta}_{\mathcal{F}}([0,T]\times\Omega,\mathbb{R}_+)$. Then, the following holds. 
\begin{enumerate}
    \item[$(a)$] The set-valued mapping 
    $$
    (t,\omega) \in [0,T] \times \Omega \rightrightarrows F(t,\omega,x(t,\omega)) \subset\mathbb{R}^n
    $$ is progressively measurable and admits a progressively measurable selection.
    \item[$(b)$] Let $(t,\omega) \in [0,T] \times \Omega \mapsto f(t,\omega)\in F(t,\omega,x(t,\omega))$ be a progressively measurable selection such that $f \in L^{\beta}_{\mathcal{F}}([0,T]\times\Omega,\mathbb{R}^n)$. Then the set-valued mapping
    $$
    (t,\omega) \in [0,T] \times \Omega \rightrightarrows T_{\overline{\textnormal{co}}F(t,\omega,x(t,\omega)))}(f(t,\omega)) \subset \mathbb{R}^n
    $$
    is progressively measurable and admits selections in $L^{\beta}_{\mathcal{F}}([0,T]\times\Omega,\mathbb{R}^n)$.
    \item[$(c)$] If for almost every $(t,\omega) \in [0,T] \times \Omega$ the sets
    $$
    F(t,\omega,x(t,\omega)) \cap \Big\{ f \in \mathbb{R}^n \, : \,  \| f - y(t,\omega) \| \le l(t,\omega) \Big\}
    $$
    are nonempty, then there exists a progressively measurable selection 
    \begin{equation*}
    (t,\omega) \mapsto f(t,\omega) \in F(t,\omega,x(t,\omega))
    \end{equation*}
    such that $ \|f(t,\omega) - y(t,\omega)\| \le l(t,\omega)$.
    \end{enumerate}
\end{thm}

\begin{rmk}[Concerning progressively measurable selections]
Observe that since $\mathcal{B}([0,T])\otimes\mathcal{G}$ endowed with the progressive $\sigma$-algebra induced by the filtration $\mathcal{F}$ is not a complete measure space, one cannot directly apply \cite[Corollary 8.2.13, Theorem 8.5.1 and Corollary 8.5.2]{Aubin1990} to derive Theorem \ref{thm:Selections}. To overcome this difficulty, one needs first to apply these latter results to the measure-theoretic completion $\overline{\mathcal{B}([0,T])\otimes\mathcal{G}}$ to obtain measurable selections, and modify them on a negligible set so that they become measurable in $\mathcal{B}([0,T])\otimes\mathcal{G}$ (see also \cite[Theorem 4.1]{Frankowska2017}).
\end{rmk}

\begin{rmk}[Shorter notation for stochastic processes]
\label{rmk:Condensed}
For the sake of conciseness, we will often drop the dependence with respect to the parameter $\omega\in\Omega$ and write $t \in [0,T] \mapsto f(t) \in F(t,x(t))$ for progressively measurable selections and maps.
\end{rmk} 

We end this preliminary section by recalling an adaptation of a general minimax theorem due to Sion \cite{Sion1958}.

\begin{thm}[Sion's minimax theorem]
\label{thm:Sion}
Let $X,Y$ be two convex subsets of Hausdorff topological spaces with $X$ being compact, and consider a continuous map $\varphi : X \times Y \to \mathbb{R}$ that is such that 
\begin{equation*}
x \in X \mapsto \varphi(x,y) \in\mathbb{R} \quad \text{is convex}
\end{equation*} 
for each $y \in Y$, and 
\begin{equation*}
y \in Y \mapsto \varphi(x,y) \in\mathbb{R} \quad \text{is concave}
\end{equation*}
for each $x \in X$. Then, it holds that 
\begin{equation*}
\min_{x \in X} \sup_{y \in Y} \varphi(x,y) = \sup_{y \in Y} \min_{x \in X} \varphi(x,y).
\end{equation*}
\end{thm}


\subsection{Stochastic Differential Inclusions}

In this section, we recollect some facts concerning set-valued stochastic dynamics. Given a progressively measurable-Lipschitz set-valued map $F : [0,T]\times\Omega\times\mathbb{R}^n \rightrightarrows \mathbb{R}^{n + d \times n}$ with nonempty compact images, we say that $x \in C^2_{\mathcal{F}}([0,T] \times \Omega,\mathbb{R}^n)$ solves the \textit{stochastic differential inclusion}
\begin{equation}
\tag{$\textnormal{SDI}$}
\label{eq:SDI}
\left\{
\begin{aligned}
\mathrm{d} x(t) & \in F(t,x(t)) \, \mathrm{d}(\lambda \times W)_t, \\
x(0) & = x_0, 
\end{aligned}
\right.
\end{equation}
if there exists a progressively measurable selection $t \in [0,T] \rightrightarrows (f(t),\sigma(t)) \in F(t,x(t))$ such that 
\begin{equation}
\label{eq:SDEDef}
\left\{
\begin{aligned}
x(t) & = f(t) \mathrm{d} t + \sigma(t) \mathrm{d}W_t, \\
x(0) & = x_0.
\end{aligned}
\right.
\end{equation}
As for deterministic differential inclusion, this class of dynamics enjoys an existence result ``à la Filippov'', which incorporates handy a priori distance estimates with respect to a given process. This is the object of the following theorem, whose proof can be established up to a small variation of the arguments proposed in \cite{DaPrato1994}.   

\begin{thm}[Filippov estimates] 
\label{theo:Filippov}
Let $F : [0,T] \times \Omega \times \mathbb{R}^n \rightrightarrows \mathbb{R}^{n + d \times n}$ be an integrably bounded and progressively measurable-Lipschitz set-valued mapping, fix $x_0,y_0 \in L^2_{\mathcal{F}_0}(\Omega,\mathbb{R}^n)$ and $(g,\zeta) \in L^2_{\mathcal{F}}([0,T]\times\Omega,\mathbb{R}^{n+n\times d})$, and consider the solution $y \in C^2_{\mathcal{F}}([0,T] \times \Omega,\mathbb{R}^n)$ of the stochastic differential equation
\begin{equation*}
\left\{
\begin{aligned}
\mathrm{d} y(t) & = g(t) \mathrm{d}t + \zeta(t) \mathrm{d} W_t, \\
y(0) & = y_0.
\end{aligned}
\right.
\end{equation*}
Moreover, suppose that the progressively measurable \textnormal{mismatch function}, defined by
\begin{equation*}
d : t \in [0,T] \mapsto \textnormal{dist}_{F(t,y(t))} \big( (g,\zeta)(t) \big) \in\mathbb{R}_+,
\end{equation*}
is an element of $L^2_{\mathcal{F}}([0,T]\times\Omega,\mathbb{R}_+)$. 

Then, there exists a solution $x \in C^2_{\mathcal{F}}([0,T] \times \Omega,\mathbb{R}^n)$ of \eqref{eq:SDI} which satisfies
\begin{equation*}
\mathbb{E} \Big[ \| x(t) - y(t) \|^2 \Big] \leq C \hspace{0.05cm} \mathbb{E} \Bigg[ \| x_0 - y_0 \|^2 + \int_0^t d^2(s) \mathrm{d}s \Bigg]
\end{equation*}
for all times $t \in[0,T]$, where the constant $C>0$ depends only on the magnitudes of the bounding map and Lipschitz constant of $F : [0,T] \times \Omega \times \mathbb{R}^d \rightrightarrows \mathbb{R}^{n + d \times n}$.
\end{thm}

In the sequel given an integrably bounded and progressively measurable-Lipschitz set-valued mapping $F : [0,T]\times\Omega\times\mathbb{R}^n \rightrightarrows \mathbb{R}^n$ along with a diffusion map $\sigma : [0,T] \times\Omega \times\mathbb{R}^n \to \mathbb{R}^n$ satisfying the relevant parts of Assumptions \ref{hyp:MAS}, we will also work with  stochastic differential inclusions of the form
\begin{equation}
\tag{$\textnormal{SDI'}$}
\label{eq:SDI'}
\left\{
\begin{aligned}
\mathrm{d} x(t) & \in F(t,x(t)) \, \mathrm{d}t + \sigma(t,x(t)) \mathrm{d} W_t, \\
x(0) & = x_0, 
\end{aligned}
\right.
\end{equation}
whose solutions are the processes $x \in C^2_{\mathcal{F}}([0,T]\times\Omega,\mathbb{R}^n)$ which solve \eqref{eq:SDEDef} for some progressively measurable selection $t \in[0,T] \mapsto f(t) \in F(t,x(t))$. Below, we recall a stochastic version of the well-known \textit{relaxation theorem} for this class of dynamics.

\begin{thm}[Relaxation] \label{theo:Relaxation}
Let $F : [0,T]\times\Omega\times\mathbb{R}^n \rightrightarrows\mathbb{R}^n$ and $\sigma:[0,T] \times \Omega\times\mathbb{R}^n \to\mathbb{R}^{d \times n}$ be integrably bounded and progressively measurable Lipschtz, fix $x_0 \in L^2_{\mathcal{F}_0}(\Omega,\mathbb{R}^n)$ and suppose that $x \in  C^2_{\mathcal{F}}([0,T]\times\Omega,\mathbb{R}^n)$ is a solution of the \textnormal{relaxed differential inclusion}
\begin{equation*}
\left\{
\begin{aligned}
\mathrm{d} x(t) & \in \overline{\textnormal{co}} F(t,x(t)) \mathrm{d}t + \sigma(t,x(t)) \mathrm{d}W_t, \\ 
x(0) & = x_0.
\end{aligned}
\right.
\end{equation*}
Then for each $\varepsilon >0$, there exists a solution $x_{\varepsilon} \in C^2_{\mathcal{F}}([0,T]\times\Omega,\mathbb{R}^n)$ of \eqref{eq:SDI'} which satisfies
\begin{equation*}
\|x - x_{\varepsilon} \|_{C^2_{\mathcal{F}}} \leq \varepsilon.  
\end{equation*}
\end{thm}

\begin{proof}
Although we did not find a satisfactory reference for this result in the literature, its proof is standard and can be carried out by following the procedure detailed e.g. in \cite[Section 2.7]{Vinter}.
\end{proof}

\begin{rmk}[Obstruction to relaxation for general stochastic inclusions]
\label{rmk:Ito}
The relaxation theorem for stochastic differential inclusions of the form \eqref{eq:SDI'} stems from Aumann's famed convexity principle for the Lebesgue -- or more generally the Bochner -- integral (see e.g. \cite[Theorem 8.6.4]{Aubin1990}). The latter asserts that, given a Borel set $I\subset [0,T]$, a real number $\beta\in[1,+\infty)$, an integrably bounded progressively measurable set-valued map $F : I \times \Omega \rightrightarrows \mathbb{R}^n$ with closed nonempty images and a progressive selection $t \in I \mapsto f(t) \in \overline{\textnormal{co}} F(t)$, there exists for each $\varepsilon > 0$ another progressively measurable selection $t \in I \mapsto f_{\varepsilon} \in F(t)$ such that 
\begin{equation}
\label{eq:Aumann}
\mathbb{E} \Bigg[ \, \bigg\| \int_I f(t) \mathrm{d} t - \int_I f_{\varepsilon}(t) \mathrm{d} t \, \bigg\|^{\beta} \, \Bigg] \leq \varepsilon. 
\end{equation}
Unfortunately, as evidenced by the following elementary counterexample, such an identity \textit{does not hold} for the It\^o integral. Indeed, consider the constant set-valued map $(t,\omega) \in [0,1] \times \Omega \rightrightarrows F \subset \mathbb{R}^2$ defined by
\begin{equation*}
F := \bigg\{ (x,y) \in[0,1]^2 : y \in[0,1-2x] ~\text{if}~ x \in[0,\tfrac{1}{2}] ~~\text{and}~~ y \in[0,2x-1] ~\text{if}~ x \in[\tfrac{1}{2},1] \bigg\},
\end{equation*}
which is clearly integrably bounded with nonempty compact images. Fixing the constant selection $t \in [0,1] \mapsto f(t) := (\tfrac{1}{2},1) \in \overline{\textnormal{co}}F(t)$, it follows from It\^o's isometry formula (see e.g. \cite[Expression (5.8)]{LeGall2016}) that %
\begin{equation*}
\mathbb{E} \Bigg[ \, \bigg\| \int_0^1 f(t) \mathrm{d} W_t - \int_0^1 f_{\varepsilon}(t) \mathrm{d} W_t \, \bigg\|^2 \, \Bigg] =  \mathbb{E} \Bigg[ \, \int_0^1 \big\| f(t)- f_{\varepsilon}(t) \big\|^2 \mathrm{d} t \, \Bigg] \geq \frac{1}{5}
\end{equation*}
for each $\varepsilon >0$ and any progressively measurable selection $t \in [0,1] \mapsto f_{\varepsilon}(t) \in F(t)$. This violates \eqref{eq:Aumann} for each $\beta \in [2,+\infty)$ by H\"older's inequality, whereas a simple contradiction argument based on both reverse dominated convergence and Egoroff theorems also yields the obstruction for $\beta\in[1,2)$. To illustrate the contrast with the Lebesgue integral, notice that in this example one can very easily find progressively measurable selections $t \in[0,T]\mapsto \tilde{f}(t) \in F(t)$ which satisfy
\begin{equation*}
\mathbb{E} \Bigg[ \, \bigg\| \int_0^1 f(t) \mathrm{d} t - \int_0^1 \tilde{f}(t) \mathrm{d} t \, \bigg\|^2 \, \Bigg] = 0,
\end{equation*}
by choosing for instance $\tilde{f}(t) := \mathds{1}_{[0,1/2]}(t)(0,1) + \mathds{1}_{[1/2,1]}(t)(1,1)$ for all times $t \in[0,1]$.
\end{rmk}


\section{Risk-Averse Optimal Control and Pontryagin Maximum Principle}
\label{sec:PMP}

\setcounter{equation}{0} \renewcommand{\theequation}{\thesection.\arabic{equation}}

In the sequel, we will investigate Pontryagin optimality conditions for the following class of risk-averse stochastic optimal control problems
\begin{equation}
\label{eq:OCP}
\tag{$\textnormal{OCP}$}
\left\{
\begin{aligned}
\ \underset{u \in \mathcal{U}}{\min} \ & \rho\big( \varphi_0(x_u(T)) \big), \\
\textnormal{s.t.} ~ \, & \mathbb{E}\big[ \varphi_i(x_u(T)) \big] \le 0, ~~ i \in \{1,\dots,\ell\}.
\end{aligned}
\right.
\end{equation}
Therein, the minimization is taken over the set of curves $x_u \in C^2_{\mathcal{F}}([0,T]\times\Omega,\mathbb{R}^n)$ solution of \eqref{eq:SDE} for some admissible control $u \in \mathcal{U}$, where 
$$
\mathcal{U} \triangleq \Big\{ u : [0,T] \times \Omega \rightarrow U : \ u \ \textnormal{is progressively measurable} \Big\}.
$$
The mapping $\rho : L^1_{\mathcal{F}_T}(\Omega,\mathbb{R}) \to \mathbb{R}$ is a finite coherent risk measure, while $\varphi_i : \Omega\times\mathbb{R}^n \rightarrow \mathbb{R}$ for $i \in\{0,\dots\ell\}$ represent a cost and functional constraints at the final time. 

From now on, we assume that the maps $f : [0,T] \times \Omega  \times \mathbb{R}^n \times U \to \mathbb{R}^n$ and $\sigma : [0,T] \times \Omega \times \mathbb{R}^n \times U \to \mathbb{R}^{d \times n}$ satisfy hypotheses \textnormal{\ref{hyp:MAS}}, and posit that the cost and constraint mappings satisfy the following assumptions. 

\vspace{10pt}

\begin{flushleft}
\begin{taggedhypmac}{(MCC)} \hfill \\
\label{hyp:MAC}
\begin{enumerate}
    \item[$(i)$] For each $i \in \{ 0 , \dots , \ell \}$ and all $x \in \mathbb{R}^n$, the mapping $\varphi_i(\cdot,x) : \Omega\to\mathbb{R}$ is $\mathcal{F}_T$-measurable and such that $\varphi_i(\cdot,0) \in L^1_{\mathcal{F}_T}(\Omega,\mathbb{R}_+)$.
    \item[$(ii)$] For every $i \in \{ 0 , \dots , \ell \}$ and almost every $\omega\in\Omega$, the application $\varphi_i(\omega,\cdot) : \mathbb{R}^n\to\mathbb{R}$ is Fr\'echet differentiable, with 
    $$
    \left\| \frac{\partial \varphi_i}{\partial x}(\omega,x) \right\| \le L
    $$
    and
    $$
    \left\| \frac{\partial \varphi_i}{\partial x}(\omega,x) - \frac{\partial \varphi_i}{\partial x}(\omega,y) \right\| \le L \| x - y \| ,
    $$
    for all $x,y \in \mathbb{R}^n$, where the constant $L > 0$ is the same as in \textnormal{\ref{hyp:MAS}}-(iii).
\end{enumerate}
\end{taggedhypmac}
\end{flushleft}

\vspace{10pt}

\begin{rmk}[On the equivalence between Bolza and Mayer problems]
It is a standard fact in optimal control theory that every Bolza problem involving a running cost can be recast as a Mayer problem in which one only minimizes a final cost. Hence, the results that we prove in this article for Mayer problems still apply to Bolza problems under appropriate assumptions. Besides, one could then relax the compactness assumption on $U\subset\mathbb{R}^m$ by simply requiring that the latter be closed, provided that the running cost satisfies a Tonelli-type growth condition with respect to the control variable. 
\end{rmk}

Throughout this article, we will use the following terminology to refer to solutions of \eqref{eq:OCP} using the following terminology. 

\begin{Def}[Admissible pairs and local minima for \eqref{eq:OCP}]
We say that $(x,u)$ is an \textnormal{admissible trajectory-control pair} for \textnormal{\eqref{eq:OCP}} if $u \in \mathcal{U}$ and $x = x_u$ is a solution of \eqref{eq:SDE} satisfying $\mathbb{E}\big[ \varphi_i(x(T)) \big] \le 0$ for all $i \in \{1,\dots,\ell\}$. Moreover, an admissible pair $(x^*,u^*)$ is a \textnormal{local minimum} for \textnormal{\eqref{eq:OCP}} if there exists $\varepsilon > 0$ such that
\begin{equation*}
\rho\big( \varphi_0(x^*(T)) \big) \le \rho\big( \varphi_0(x(T)) \big),
\end{equation*}
for every other admissible pair $(x,u)$ satisfying $\| x - x^* \|_{C^2_{\mathcal{F}}} \le \varepsilon$.
\end{Def}
From now on, we assume the existence of a local minimum for \eqref{eq:OCP}, denoted $(x^*,u^*)$.

We are now ready to state and prove our main result, which are first-order necessary optimality conditions for \textnormal{\eqref{eq:OCP}} in the form of a Pontryagin Maximum Principle. In what follows, we denote by $H : [0,T] \times \Omega \times \mathbb{R}^n \times U \times \mathbb{R}^n \times \mathbb{R}^{n \times d} \rightarrow \mathbb{R}$ the \textit{Hamiltonian} associated with \textnormal{\eqref{eq:OCP}}, defined by
\begin{equation}
\label{eq:Hamiltonian}
H(t,\omega,x,u,p,q) \triangleq p \cdot f(t,\omega,x,u) + \sum^d_{i=1} q_i \cdot \sigma_i(t,\omega,x,u) .
\end{equation}
for all $(t,\omega,x,u,p,q) \in [0,T] \times \Omega \times \mathbb{R}^n \times U \times \mathbb{R}^n \times \mathbb{R}^{n \times d}$. We also consider the set of \textit{active indices} at $x^*(T)$, which is given by
$$
I^{\circ}(x^*(T)) \triangleq \Big\{ i \in \{ 1,\dots,\ell \} \, : \, \mathbb{E}\big[ \varphi_i(x^*(T)) \big] = 0 \Big\} .
$$
Finally, for the sake of clarity in the exposition, we separate the cases of controlled and uncontrolled diffusions, as the latter can be proven under milder assumptions.


\subsection{The PMP with Controlled Diffusion} 
\label{sec:controlledDiff}

In the case where the control variable acts on both the drift and the diffusion terms, we need to supplement hypotheses \ref{hyp:MAS} and \ref{hyp:MAC} with the following assumption.

\vspace{10pt}

\begin{flushleft}
\begin{taggedhypacd}{(ACD)} \hfill \\
\label{hyp:ACD}
The stochastic drift $f : [0,T] \times \Omega \times \mathbb{R}^n \times U \rightarrow \mathbb{R}^n$ and the diffusion term $\sigma : [0,T] \times \Omega \times \mathbb{R}^n \times U \rightarrow \mathbb{R}^{n \times d}$ are such that the velocity sets, defined by
$$
F (t,\omega,x) \triangleq \Big\{ \big( f(t,\omega,x,u) , \sigma(t,\omega,x,u) \big) : \ u \in U \Big\} \subset \mathbb{R}^{n + n \times d}, 
$$
are convex for almost every $(t,\omega) \in [0,T] \times \Omega$ and all $x\in \mathbb{R}^n$.
\end{taggedhypacd}
\end{flushleft}

\vspace{5pt}
    
\begin{rmk}
The above assumption, which has already been considered in \cite{Frankowska2020} in a similar setting, is standard in deterministic optimal control, where it is very useful to guarantee the existence of optimal controls. In particular, \textnormal{\ref{hyp:ACD}} holds true e.g. when $f$ and $\sigma$ are affine in the control variable and $U$ is convex.
\end{rmk}

\begin{thm}[Risk-averse PMP for \eqref{eq:OCP} with controlled diffusion] \label{thm:controlledDiff}
Suppose that hypotheses \textnormal{\ref{hyp:MAS}}, \textnormal{\ref{hyp:MAC}}, and \textnormal{\ref{hyp:ACD}} hold, and let $(x^*,u^*)$ be a local minimum for \textnormal{\eqref{eq:OCP}}. Then there exists a risk parameter $\xi^* \in \partial \rho\big( \varphi_0(x^*(T)) \big)$, non-trivial Lagrange multipliers $(\mathfrak{p}_0,\dots,\mathfrak{p}_{\ell}) \in \{-1,0\} \times \mathbb{R}_-^\ell$ and a pair of stochastic processes $(p^*,q^*) \in C^2_{\mathcal{F}}([0,T] \times \Omega,\mathbb{R}^n) \times L^2_{\mathcal{F}}([0,T] \times \Omega,\mathbb{R}^{n \times d})$ such that the following holds. 
\begin{enumerate}
\item[$(i)$] The complementary slackness conditions 
\begin{equation}
\label{eq:PMPSlack}
\mathfrak{p}_i \mathbb{E}[\varphi_i(x^*(T))] = 0
\end{equation}
are satisfied for each $i \in\{1,\dots,\ell\}$.
\item[$(ii)$] The risk parameter $\xi^* \in \partial\rho\big( \varphi_0(x^*(T)) \big)$ is characterised by the condition
\begin{equation}
\label{eq:PMPMaximisationOther}
\mathbb{E}[ \xi^* \varphi_0(x^*(T)) ] = \underset{\xi \in \partial \rho(0)}{\max} \ \mathbb{E}[ \xi \varphi_0(x^*(T)) ].
\end{equation}
\item[$(iii)$] The processes $(p^*,q^*) \in C^2_{\mathcal{F}}([0,T]\times\Omega,\mathbb{R}^n) \times L^2_{\mathcal{F}}([0,T] \times\Omega,\mathbb{R}^{d \times n})$ solve the backward adjoint equations
\begin{equation}
\label{eq:PMPAdjoint}
\left\{
\begin{aligned}
& \mathrm{d} p^*(t) = -\frac{\partial H}{\partial x} \big(t,x^*(t),u^*(t),p^*(t),q^*(t) \big) \mathrm{d}t + q^*(t) \mathrm{d} W_t \\
& p^*(T) = \xi^* \mathfrak{p}_0 \nabla\varphi_0(x^*(T)) + \sum_{i=1}^{\ell} \mathfrak{p}_i \nabla\varphi_i(x^*(T)).
\end{aligned}
\right.
\end{equation} 
\item[$(iv)$] The Pontryagin maximization condition 
\begin{equation}
\label{eq:PMPMaximisation}
H \big(t,x^*(t),u^*(t),p^*(t),q^*(t) \big) = \max_{u \in U} \ H\big(t,x^*(t),u,p^*(t),q^*(t) \big)
\end{equation}
holds almost everywhere. 
\end{enumerate}
Furthermore, if there exists a solution $y_{g_1,g_2} \in C^2_{\mathcal{F}}([0,T] \times\Omega,\mathbb{R}^n)$ of the linearized dynamics \eqref{eq:LSDE}  (see Step 1 below) that is such that
\begin{equation*}
\mathbb{E} \Big[ \nabla \varphi_i(x^*(T)) \cdot y_{g_1,g_2}^*(T) \Big] < 0    
\end{equation*}
for every $i \in I^{\circ}(x^*(T))$, then the PMP is \textnormal{normal}, i.e. $\mathfrak{p}_0= -1$. 
\end{thm}

We split the proof of Theorem \ref{thm:controlledDiff} into five steps. In Step 1, we start by introducing a class of set-valued linearizations along candidate optimal trajectory-control pairs. We subsequently perform a separation argument on the reachable set of the corresponding linearized system and the linearizing cone to the constraints, first in the absence of qualification conditions in Step 2, and then when the constraints are qualified in Step 3. We further show in Step 4 that one can in fact select an optimal risk parameter for which the variational inequalities hold uniformly with respect to the whole reachable set, and finally conclude in Step 5 by proving that these latter yield the PMP in conjunction with the adjoint dynamics. 

In what follows, we will almost systematically use the convention introduced in Remark \ref{rmk:Condensed} for stochastic processes, and drop all explicit dependence in the variable $\omega\in\Omega$ unless necessary.

\vspace{5pt}

\noindent \textbf{Step 1 -- Variational linearizations along $\bm{(x^*,u^*)}$.} For every $(t,\omega,x,u) \in [0,T] \times \Omega \times \mathbb{R}^n \times U$, we introduce the notation 
$$
(f,\sigma)(t,\omega,x,u) \triangleq \big( f(t,\omega,x,u) , \sigma(t,\omega,x,u) \big)
$$ 
and recall following hypotheses \ref{hyp:ACD} that the set
\begin{align*}
    F (t,\omega,x) = \Big\{ (f,\sigma)(t,\omega,x,u) \, : \, u \in U \Big\} \subset \mathbb{R}^{n + n \times d}
\end{align*}
is convex. Besides, under hypotheses \ref{hyp:MAC}, one can easily prove that the set-valued mapping $F : [0,T] \times \Omega \times \mathbb{R}^n \rightrightarrows \mathbb{R}^{n + n \times d}$ is integrably bounded as well as progressively measurable-Lipschitz with nonempty compact images, following e.g. \cite[Theorem 8.2.8]{Aubin1990}. In particular, using the condensed notation of Remark \ref{rmk:Condensed}, it holds that 
\begin{equation*}
t \in [0,T] \mapsto (f,\sigma)(t,x^*(t),u^*(t)) \in F(t,x(t)),  
\end{equation*}
is an element of $L^2_{\mathcal{F}}([0,T] \times \Omega,\mathbb{R}^{n + n \times d})$. Moreover, it follows from Theorem \ref{thm:Selections} that the progressively measurable set-valued map
\begin{equation*}
t \in [0,T] \mapsto T_{F(t,x^*(t))} \big( (f,\sigma)(t,x^*(t),u^*(t)) \big) ,
\end{equation*}
has nonempty compact and convex images, and thus admits progressive selections
$$
t \in [0,T] \mapsto (g_1,g_2)(t) \in T_{F(t,x^*(t))} \big( (f,\sigma)(t,x^*(t),u^*(t)) \big)
$$
which belong to $L^2_{\mathcal{F}}([0,T] \times \Omega,\mathbb{R}^{n + n \times d})$.

Given such a progressively measurable tangent selection $(g_1,g_2)$, we denote by $y_{g_1,g_2} \in C^2_{\mathcal{F}}([0,T] \times \Omega,\mathbb{R}^n)$ the unique (up to stochastic indistinguishability) solution of the linearized stochastic differential equation
\begin{equation}
\label{eq:LSDE}
\tag{$\textnormal{LSDE}_{g_1,g_2}$}
\left\{
\begin{aligned}
\mathrm{d}y(t) & = \Big( A(t) y(t) + g_1(t) \Big) \mathrm{d}t + \sum^d_{i=1} \Big( D_i(t) y(t) + g_2^i(t) \Big) \mathrm{d}W^i_t , \\
y(0) & = 0,
\end{aligned}
\right.
\end{equation}
in which we used the condensed notations
$$
A(t) \triangleq \frac{\partial f}{\partial x}(t,x^*(t),u^*(t)) \qquad \text{and} \qquad D_i(t) \triangleq \frac{\partial \sigma_i}{\partial x}(t,x^*(t),u^*(t)) ,
$$
for almost every $t \in[0,T]$ and each $i \in\{ 1,\dots,d\}$. In the following lemma, we prove that $y_{g_1,g_2}$ is continuous with respect to $(g_1,g_2)$ in the strong $L^2_{\mathcal{F}}$-topology. This result will be useful later on in the proof of the maximum principle. 

\begin{lem} \label{lemma:contY}
There exists a constant $C > 0$ depending only on the magnitudes of $T,\|k\|_{L^2_\mathcal{F}}$ and $L$ such that for any given pair of progressively measurable selections
$$
t \in [0,T] \mapsto (g_1,g_2)(t),(\tilde g_1,\tilde g_2)(t) \in T_{F(t,x^*(t))} \big( (f,\sigma)(t,x^*(t),u^*(t)) \big),
$$
it holds that
$$
\| y_{g_1,g_2} - y_{\tilde g_1,\tilde g_2} \|_{C^2_{\mathcal{F}}} \le C \| (g_1,g_2) - (\tilde g_1,\tilde g_2) \|_{L^2_{\mathcal{F}}} .
$$
\end{lem}

\begin{proof}
Thanks to hypotheses \ref{hyp:MAS} and a routine application of Burkholder-Davis-Gundy's and H\"older's inequalities, we obtain for every $t \in [0,T]$ that 
{\small
\begin{align*}
    &\mathbb{E}\left[ \underset{s \in [0,t]}{\sup} \| y_{g_1,g_2}(s) - y_{\tilde g_1,\tilde g_2}(s) \|^2 \right] \\
    &\le C \, \mathbb{E}\left[ \left( \int^t_0 \| A(s) \| \, \| y_{g_1,g_2}(s) - y_{\tilde g_1,\tilde g_2}(s) \| \mathrm{d}s \right)^2 \, \right] \\
    & \hspace{0.4cm} + C \sum^d_{i=1} \mathbb{E}\left[ \int^t_0 \| D_i(s) \|^2 \, \| y_{g_1,g_2}(s) - y_{\tilde g_1,\tilde g_2}(s) \|^2 \mathrm{d}s \right] \\
    & \hspace{0.4cm} + C \, \mathbb{E}\left[ \left( \int^t_0 \| g_1(s) - \tilde g_1(s) \| \; \mathrm{d}s \right)^2 + \sum^d_{i=1} \int^t_0 \| g_2^i(s) - \tilde g_2^i(s) \|^2 \; \mathrm{d}s \right] \\
    &\le C \mathbb{E}\left[ \int^t_0 \underset{\zeta \in [0,s]}{\sup} \| y_{g_1,g_2}(\zeta) - y_{\tilde g_1,\tilde g_2}(\zeta) \|^2  \mathrm{d}s + \int^T_0 \| (g_1,g_2)(s) - (\tilde g_1,\tilde g_2)(s) \|^2  \mathrm{d}s \right],
\end{align*}}
where $C > 0$ denotes some overloaded constant which only depends on the magnitudes of $T,\|k\|_{L^2_{\mathcal{F}}}$ and $L$. We then conclude by an application of Gronw\"all's lemma.
\end{proof}

In this context, we have the following fundamental linearization result.

\begin{thm}[Variational linearization]
\label{theo:var}
For any progressively measurable selection $t \in [0,T] \mapsto (g_1,g_2)(t) \in T_{F(t,x^*(t))} (f,\sigma)(t,x^*(t),u^*(t))$ and each $\varepsilon > 0$, there exists a solution $x^{\varepsilon}_{g_1,g_2} \in C^2_{\mathcal{F}}([0,T] \times \Omega,\mathbb{R}^n)$ of the dynamics \eqref{eq:SDI} such that 
\begin{equation}
\label{eq:var}
\displaystyle \underset{\varepsilon \rightarrow 0^+}{\lim} \frac{1}{\varepsilon} \mathbb{E} \bigg[ \, \underset{t \in [0,T]}{\sup} \| x^{\varepsilon}_{g_1,g_2}(t) - x^*(t) - \varepsilon y_{g_1,g_2}(t) \| \bigg] = 0,
\end{equation}
where $y_{g_1,g_2} \in C^2_{\mathcal{F}}([0,T] \times \Omega,\mathbb{R}^n)$ is the unique solution of \eqref{eq:LSDE}.
\end{thm}

\begin{proof}
Our proof is inspired from that of \cite[Theorem 3.12]{Bonnet2021}. We fix a progressively measurable selection $t \in [0,T] \mapsto (g_1,g_2)(t) \in T_{F(t,x^*(t))} (f,\sigma)(t,x^*(t),u^*(t))$, some $\varepsilon > 0$ and consider the progressively measurable mapping 
\begin{equation*}
    t \in[0,T] \mapsto d_{\varepsilon}(t) := \textnormal{dist}_{F(t,x^*(t))}\Big( (f,\sigma)(t,x^*(t),u^*(t)) + \varepsilon (g_1,g_2)(t) \Big).
\end{equation*}
It can be checked that the latter satisfies 
\begin{equation*}
d_{\varepsilon}(t,\omega) \le \varepsilon \| (g_1,g_2)(t,\omega) \| \qquad \text{and} \qquad \underset{\varepsilon \rightarrow 0^+}{\lim} \ d_{\varepsilon}(t,\omega) / \varepsilon = 0
\end{equation*}
for almost every $(t,\omega) \in [0,T]\times\Omega$, so that in particular $d_{\varepsilon} \in L^2_{\mathcal{F}}([0,T]\times\Omega,\mathbb{R})$. Let $\bar x^{\varepsilon}_{g_1,g_2} \in C^2_{\mathcal{F}}([0,T]\times\Omega,\mathbb{R}^n)$ be the unique (up to stochastic indistinguishability) solution of
\begin{equation}
\tag{$\textnormal{SDI}^{\varepsilon}_{g_1,g_2}$}
\label{eq:SDIEps}
\left\{
\begin{aligned}
\mathrm{d}x(t) & = \Big( f \big( t,x(t),u^*(t) \big) + \varepsilon g_1(t) \Big) \mathrm{d}t \\
& \hspace{0.45cm} + \Big( \sigma \big(t,x(t),u^*(t) \big) + \varepsilon g_2(t) \Big) \mathrm{d}W_t, \\
x(0) & = x_0.
\end{aligned}
\right.
\end{equation}
Thanks to hypotheses \ref{hyp:MAS} and a routine application of the Burkholder-Davis-Gundy and H\"older inequalities, we readily obtain that for every $t \in [0,T]$, it holds  
{\small
\begin{align*}
    &\mathbb{E}\left[ \underset{s \in [0,t]}{\sup} \| \bar x^{\varepsilon}_{g_1,g_2}(s) - x^*(s) \|^2 \right] \\
    & \le C \varepsilon^2 \mathbb{E}\left[ \left( \int^t_0 \| g_1(s) \| \; \mathrm{d}s \right)^2 + \sum^d_{i=1} \int^t_0 \| g_2(s)^i \|^2 \; \mathrm{d}s \right] \\
    & \hspace{0.25cm} + C \mathbb{E}\left[ \left( \int^t_0 \left\| \int^1_0 \frac{\partial f}{\partial x} \Big( s,x^*(s) + \theta (\bar{x}^{\varepsilon}_{g_1,g_2}(s)  - x^*(s)),u^*(s) \Big) \big( \bar x^{\varepsilon}_{g_1,g_2}(s) - x^*(s) \big) \mathrm{d}\theta \, \right\| \mathrm{d}s \right)^2 \right] \\
    & \hspace{0.25cm} + C \sum^d_{i=1} \mathbb{E}\left[ \int^t_0 \left\| \int^1_0 \frac{\partial \sigma^i}{\partial x} \Big( s,x^*(s) + \theta (\bar{x}^{\varepsilon}_{g_1,g_2}(s)  - x^*(s)),u^*(s) \Big) \big( \bar x^{\varepsilon}_{g_1,g_2}(s) - x^*(s) \big) \mathrm{d}\theta \right\|^2 \hspace{-0.1cm} \mathrm{d}s \right] \\
    & \le C \left( \mathbb{E}\left[ \int^t_0 \underset{\zeta \in [0,s]}{\sup} \| \bar x^{\varepsilon}_{g_1,g_2}(\zeta) - x^*(\zeta) \|^2 \; \mathrm{d}s + \varepsilon^2 \int^T_0 \| (g_1,g_2)(s) \|^2 \; \mathrm{d}s \right] \right) ,
\end{align*}}
where $C > 0$ denotes some overloaded constant which only depends on $T$ and $L$. Then, a direct application of Gronw\"all's inequality leads to
\begin{equation} 
\label{eq:convTraj}
\mathbb{E}\left[ \underset{t \in [0,T]}{\sup} \| \bar x^{\varepsilon}_{g_1,g_2}(t) - x^*(t) \|^2 \right] \le C \varepsilon^2 \| (g_1,g_2) \|^2_{L^2_{\mathcal{F}}} .
\end{equation}
On the other hand, by introducing the notations
{\small
\begin{equation*}
\left\{
\begin{aligned}
A^{\varepsilon}_{g_1,g_2}(t) &  \triangleq \int^1_0 \left( \frac{\partial f}{\partial x} \Big( s,x^*(s) + \theta (\bar{x}^{\varepsilon}_{g_1,g_2}(s)  - x^*(s)),u^*(s) \Big)  - \frac{\partial f}{\partial x}(s,x^*(s),u^*(s)) \right) \mathrm{d}\theta, \\
D^{\varepsilon,i}_{g_1,g_2}(t) & \triangleq \int^1_0 \left( \frac{\partial \sigma^i}{\partial x} \Big( s,x^*(s) + \theta (\bar{x}^{\varepsilon}_{g_1,g_2}(s)  - x^*(s)),u^*(s) \Big)  - \frac{\partial \sigma^i}{\partial x}(s,x^*(s),u^*(s)) \right) \mathrm{d}\theta,
\end{aligned}
\right.
\end{equation*}
}
for all times $t \in[0,T]$, one may easily show that the process defined by 
\begin{equation*}
r^{\varepsilon}_{g_1,g_2}(t) \; \triangleq \; \bar x^{\varepsilon}_{g_1,g_2}(t) - x^*(t) - \varepsilon y_{g_1,g_2}(t)
\end{equation*}
solves the stochastic differential equation
\begin{equation*}
\left\{
\begin{aligned}
    \mathrm{d}r(t) &= \Big( A(t) + A^{\varepsilon}_{g_1,g_2}(t) \Big) r(t) \mathrm{d}t + \sum^d_{i=1} \Big( D_i(t) + D^{\varepsilon,i}_{g_1,g_2}(t) \Big) r(t) \mathrm{d}W^i_t \\
    & \hspace{1.375cm} + \varepsilon \left( A^{\varepsilon}_{g_1,g_2}(t) y_{g_1,g_2}(t) \; \mathrm{d}t + \sum^d_{i=1} D^{\varepsilon,i}_{g_1,g_2}(t) y_{g_1,g_2}(t) \mathrm{d}W^i_t \right), \\
    r(0) & = 0. 
\end{aligned}
\right.
\end{equation*}
Thanks to hypotheses \ref{hyp:MAS}, it then follows from Lemma \ref{lemma:bound} applied to the latter dynamics that
{\small
\begin{align*}
    &\mathbb{E}\left[ \underset{t \in [0,T]}{\sup} \big\| \bar x^{\varepsilon}_{g_1,g_2}(t) - x^*(t) - \varepsilon y_{g_1,g_2}(t) \big\| \right] \\
    &\le \varepsilon C \left( \mathbb{E}\left[ \int^T_0 \| A^{\varepsilon}_{g_1,g_2}(s) \| \| y_{g_1,g_2}(s) \| \mathrm{d}s \right] + \sum^d_{i=1} \mathbb{E}\left[ \int^T_0 \| D^{\varepsilon,i}_{g_1,g_2}(s) \|^2 \| y_{g_1,g_2}(s) \|^2 \mathrm{d}s \right]^{1/2} \right) .
\end{align*}}
where $C > 0$ denotes some overloaded constant which only depends on the magnitudes of $T,\|k\|_{L^2_{\mathcal{F}}}$ and $L$. Observe that now that, from \eqref{eq:convTraj}, we may infer that
$$
\underset{t \in [0,T]}{\sup} \| \bar x^{\varepsilon}_{g_1,g_2}(t) - x^*(t) \|^2 \underset{\varepsilon \to 0^+}{\longrightarrow 0},
$$
almost surely. From hypothesis \ref{hyp:MAS} and the dominated convergence, we thus have
{\small
$$
\mathbb{E}\left[ \int^T_0 \| A^{\varepsilon}_{g_1,g_2}(s) \| \, \| y_{g_1,g_2}(s) \| \; \mathrm{d}s \right] + \sum^d_{i=1} \mathbb{E}\left[ \int^T_0 \| D^{\varepsilon,i}_{g_1,g_2}(s) \|^2 \, \| y_{g_1,g_2}(s) \|^2 \; \mathrm{d}s \right]^{1/2} \underset{\varepsilon \to 0^+}{\longrightarrow} 0,
$$}
which allows us to conclude that
$$
\underset{\varepsilon \rightarrow 0^+}{\lim} \ \frac{1}{\varepsilon} \mathbb{E} \bigg[\, \underset{t \in [0,T]}{\sup} \| \bar x^{\varepsilon}_{g_1,g_2}(t) - x^*(t) - \varepsilon y_{g_1,g_2}(t) \| \bigg] = 0.
$$
To end the proof of our claim, there remains to establish the existence of a solution $x^{\varepsilon}_{g_1,g_2} \in C^2_{\mathcal{F}}([0,T] \times \Omega,\mathbb{R}^n)$ to \eqref{eq:SDI} which satisfies
\begin{equation}
\label{eq:Linearization1}
\underset{\varepsilon \rightarrow 0^+}{\lim} \frac{1}{\varepsilon^2} \mathbb{E}\left[ \underset{t \in [0,T]}{\sup} \| x^{\varepsilon}_{g_1,g_2}(t) - \bar x^{\varepsilon}_{g_1,g_2}(t) \|^2 \right] = 0.
\end{equation}
By Theorem \ref{thm:Selections}, there exists for every $\varepsilon > 0$ a progressively measurable selection $t \in [0,T] \mapsto (h^{\varepsilon}_1,h^{\varepsilon}_2)(t) \in F(t,x^*(t))$ which is such that
\begin{equation*}
\| (f,\sigma)(t,x^*(t),u^*(t)) + \sqrt{\varepsilon} (g_1,g_2)(t) - (h^{\varepsilon}_1,h^{\varepsilon}_2)(t) \| = d_{\sqrt{\varepsilon}}(t),
\end{equation*}
almost everywhere. Therefore, the progressively measurable maps defined by 
$$
t \in[0,T] \mapsto (g^{\varepsilon}_1,g^{\varepsilon}_2)(t) \triangleq \frac{(h^{\varepsilon}_1,h^{\varepsilon}_2)(t) - (f,\sigma)(t,x^*(t),u^*(t))}{\sqrt{\varepsilon}}
$$
are elements of $L^2_{\mathcal{F}}([0,T] \times \Omega,\mathbb{R}^{n + n \times d})$ since they are bounded almost everywhere in norm by $2 \|(g_1,g_2)\|$ $+1$, and are such that
\begin{equation}
\label{eq:Construction1}
(f,\sigma)(t,x^*(t),u^*(t)) + \sqrt{\varepsilon} (g^{\varepsilon}_1,g^{\varepsilon}_2)(t) \in F(t,x^*(t)).
\end{equation}
Moreover, it can be easily checked that since $d_{\sqrt{\varepsilon}}(t)/\sqrt{\varepsilon} \to 0^+$ as $\varepsilon \to0^+$, one has  
\begin{equation}
\label{eq:Linearization2}
\| (g_1,g_2) - (g^{\varepsilon}_1,g^{\varepsilon}_2) \|_{L^2_{\mathcal{F}}} ~\underset{\varepsilon \to0^+}{\longrightarrow}~ 0
\end{equation}
by Lebesgue's dominated convergence theorem. Similarly, by Theorem \ref{thm:Selections} combined with \eqref{eq:Construction1}, one can find a selection $t \in [0,T] \mapsto (\kappa^{\varepsilon}_1,\kappa^{\varepsilon}_2)(t) \in F(t,\bar x^{\varepsilon}_{g_1,g_2}(t))$ for which
\begin{equation}
\label{eq:Construction2}
\begin{aligned}
    \big\| (f,\sigma)(t,x^*(t),u^*(t)) & + \sqrt{\varepsilon}  \big( g^{\varepsilon}_1,g^{\varepsilon}_2)(t) - (\kappa^{\varepsilon}_1,\kappa^{\varepsilon}_2)(t) \big\| \\
    & \hspace{0cm} = \textnormal{dist}_{F(t,\bar x^{\varepsilon}_{g_1,g_2}(t))} \Big( (f,\sigma)(t,x^*(t),u^*(t)) + \sqrt{\varepsilon} (g^{\varepsilon}_1,g^{\varepsilon}_2)(t) \Big) \\
    & \hspace{0cm} \le d_{\mathcal{H}} \Big( F(t,x^*(t)) , F(t,\bar x^{\varepsilon}_{g_1,g_2}(t)) \Big) \\
    & \hspace{0cm} \le L \, \| \bar x^{\varepsilon}_{g_1,g_2}(t) - x^*(t) \|
\end{aligned}
\end{equation}
holds almost everywhere. At this stage, thanks to the convexity requirement formulated in hypothesis \ref{hyp:ACD}, one can further observe that
\begin{equation}
\label{eq:Construction3}
(1 - \sqrt{\varepsilon}) (f,\sigma) (t,\bar x^{\varepsilon}_{g_1,g_2}(t),u^*(t)) + \sqrt{\varepsilon} (\kappa^{\varepsilon}_1,\kappa^{\varepsilon}_2)(t) \in F \big( t,\bar x^{\varepsilon}_{g_1,g_2}(t) \big) ,
\end{equation}
which implies in particular that
$$
\textnormal{dist}_{F(\cdot,\bar x^{\varepsilon}_{g_1,g_2}(\cdot))} \Big( (f,\sigma)(\cdot,\bar x^{\varepsilon}_{g_1,g_2}(\cdot),u^*(\cdot)) + \varepsilon (g_1,g_2)(\cdot) \Big) \in L^2_{\mathcal{F}}([0,T] \times \Omega,\mathbb{R}) .
$$
Since $\bar{x}^{\varepsilon}_{g_1,g_2} \in C^2_{\mathcal{F}}([0,T]\times\Omega,\mathbb{R}^n)$ solves \eqref{eq:SDIEps}, we can apply Theorem \ref{theo:Filippov} to obtain the existence of a solution $x^{\varepsilon}_{g_1,g_2} \in C^2_{\mathcal{F}}([0,T] \times \Omega,\mathbb{R}^n)$ to \eqref{eq:SDI} which satisfies
\begin{align*}
    &\mathbb{E} \bigg[ \, \underset{t \in [0,T]}{\sup} \ \big\| x^{\varepsilon}_{g_1,g_2}(t) - \bar x^{\varepsilon}_{g_1,g_2}(t) \big\|^2 \bigg] \\
    & \hspace{1cm} \le C \mathbb{E} \Bigg[ \int^T_0 \textnormal{dist}_{F(t,\bar x^{\varepsilon}_{g_1,g_2}(t))} \Big( (f,\sigma)(t,\bar x^{\varepsilon}_{g_1,g_2}(t),u^*(t)) + \varepsilon (g_1,g_2)(t) \Big)^2 \; \mathrm{d}t \Bigg] .
\end{align*}
This last identity together with the convergence results of \eqref{eq:Linearization1}-\eqref{eq:Linearization2} and the constructions detailed in \eqref{eq:Construction2}-\eqref{eq:Construction3} allows us to finally recover that
\begin{align*}
    &\frac{1}{\varepsilon^2} \mathbb{E}\left[ \int^T_0 \textnormal{dist}_{F(t,\bar x^{\varepsilon}_{g_1,g_2}(t))} \Big( (f,\sigma) \big(t,\bar x^{\varepsilon}_{g_1,g_2}(t),u^*(t)\big) + \varepsilon (g_1,g_2)(t) \Big)^2 \; \mathrm{d}t \right] \\
    &\le \frac{1}{\varepsilon^2} \mathbb{E}\Bigg[ \int^T_0 \Big\| (f,\sigma)(t,\bar x^{\varepsilon}_{g_1,g_2}(t),u^*(t)) + \varepsilon (g_1,g_2)(t) \\
    &\qquad \qquad \qquad \qquad - (1 - \sqrt{\varepsilon}) (f,\sigma)(t,\bar x^{\varepsilon}_{g_1,g_2}(t),u^*(t)) - \sqrt{\varepsilon} (\kappa^{\varepsilon}_1,\kappa^{\varepsilon}_2)(t) \Big\|^2 \mathrm{d}t \Bigg] \\
    &\le C \mathbb{E}\left[ \int^T_0 \| (g_1,g_2)(t) - (g^{\varepsilon}_1,g^{\varepsilon}_2)(t) \|^2 \; \mathrm{d}t \right] \\
    &\quad + \frac{C}{\varepsilon} \mathbb{E}\left[ \int^T_0 \big\| (f,\sigma)(t,\bar x^{\varepsilon}_{g_1,g_2}(t),u^*(t)) - (f,\sigma)(t,x^*(t),u^*(t)) \big\|^2 \; \mathrm{d}t \right] \\
    &\quad + \frac{C}{\varepsilon} \mathbb{E}\left[ \int^T_0 \big\| (f,\sigma)(t,x^*(t),u^*(t)) + \sqrt{\varepsilon} \big( g^{\varepsilon}_1,g^{\varepsilon}_2)(t) - (\kappa^{\varepsilon}_1,\kappa^{\varepsilon}_2)(t) \big) \big\|^2 \; \mathrm{d}t \right] \\
    &\le C \mathbb{E}\left[ \int^T_0 \| (g_1,g_2)(t) - (g^{\varepsilon}_1,g^{\varepsilon}_2)(t) \|^2 \; \mathrm{d}t \right] + \frac{C}{\varepsilon} \mathbb{E}\left[ \underset{t \in [0,T]}{\sup} \| \bar x^{\varepsilon}_{g_1,g_2}(t) - x^*(t) \|^2 \right] \underset{\varepsilon \to 0^+}{\longrightarrow} 0,
\end{align*}
thanks to \eqref{eq:convTraj}, where $C > 0$ is some overloaded constant which only depends on the magnitudes of $T,\|k\|_{L^2_{\mathcal{F}}}$ and $L$, from which the thesis follows.
\end{proof}


\noindent \textbf{Step 2 -- Separation theorem without constraint qualification.} From now on, up to relabeling the indices, we assume without loss of generality that there exists $j \in \{ 1,\dots,\ell \}$ such that $I^{\circ}(x^*(T)) = \{ 1,\dots,j \}$. In addition, we posit that 
\begin{equation*}
\nabla \varphi_0(x^*(T)) \neq 0 \qquad \text{and} \qquad \nabla \varphi_i(x^*(T)) \neq 0 ~~ \text{for every $i \in I^{\circ}(x^*(T))$}
\end{equation*}
as elements of $L^2_{\mathcal{F}_T}(\Omega,\mathbb{R}^n)$. Otherwise, if $\nabla \varphi_k(x^*(T)) = 0$ for some $k \in \{0,\dots,j\}$, one can observe that the statements of Theorem \ref{thm:controlledDiff} are trivially satisfied with $\mathfrak{p}_k = 1$, $\mathfrak{p}_i =0$ for $i \in\{0,\dots,\ell\} \setminus\{k\}$, $p^*,q^*$being set to zero, and $\xi^*$ being a solution of \eqref{eq:PMPMaximisationOther}.

By leveraging the notation introduced hereinabove, we define the reachable set of the linearized Cauchy problem \eqref{eq:LSDE} by 
{\small
\begin{equation*}
\begin{aligned}
\mathcal{R}_T \triangleq \bigg\{ y_{g_1,g_2}(T) \in L^2_{\mathcal{F}_T}(\Omega,\mathbb{R}^n) & \, : \, \text{$y_{g_1,g_2} \in C^2_{\mathcal{F}}([0,T]\times\Omega,\mathbb{R}^n)$ solves \eqref{eq:LSDE} for some} \\
& \hspace{0.05cm} ~ t \in [0,T] \mapsto (g_1,g_2)(t) \in T_{F(t,x^*(t))} \big( (f,\sigma)(t,x^*(t),u^*(t)) \big) \bigg\} .
\end{aligned}
\end{equation*}}
Since the images of $T_{F(\cdot,x^*(\cdot))} (f,\sigma)(\cdot,x^*(\cdot),u^*(\cdot))$ are convex cones and \eqref{eq:LSDE} is linear with respect to both $y_{g_1,g_2}$ and $(g_1,g_2)$, one can easily check that $\mathcal{R}_T \subset L^2_{\mathcal{F}_T}(\Omega,\mathbb{R}^n)$ is a nonempty convex cone as well. At this stage, we introduce the set
\begin{align*}
    \mathcal{B}_T \triangleq \bigg\{ \hspace{-0.1cm} \Big( &\mathbb{E}\big[ \nabla \varphi_1(x^*(T)) \cdot y_{g_1,g_2}(T) \big] , \dots  , \mathbb{E}\big[ \nabla \varphi_j(x^*(T)) \cdot  y_{g_1,g_2}(T) \big] \Big) \hspace{-0.1cm} : \hspace{-0.05cm} y_{g_1,g_2}(T) \in \mathcal{R}_T \hspace{-0.1cm} \bigg\} ,
\end{align*}
which is a nonempty convex cone in $\mathbb{R}^j$, and assume at first that 
$$
\mathcal{B}_T \cap (-\infty,0)^j = \emptyset.
$$
In that case, by the  separation theorem, we may infer the existence of a non-trivial element $\mathfrak{p} \in \mathbb{R}^j \setminus \{0\}$ such that
$$
-\infty < \underset{a \in (-\infty,0)^j}{\sup} \ \mathfrak{p} \cdot a \; \le \; \underset{b \in \mathcal{B}_T}{\inf} \ \mathfrak{p} \cdot b < \infty .
$$
Observing that both $\mathcal{B}_T$ and $(-\infty,0)^j$ are cones while using standard results of convex analysis, the latter separation inequality implies that
\begin{equation} 
\label{eq:var1}
    \sum^j_{i=1} \mathfrak{p}_i \mathbb{E}\big[ \nabla \varphi_i(x^*(T)) \cdot y_{g_1,g_2}(T) \big] \ge 0 \quad \text{and} \quad \text{$\mathfrak{p}_i \ge 0$ ~ for each  $i\in\{1,\dots,j\}$}.
\end{equation}

\vspace{5pt}


\noindent \textbf{Step 3 -- Separation theorem with constraint qualification.} We now investigate the scenario in which $\mathcal{B}_T \cap (-\infty,0)^j \neq \emptyset$, which calls for a deeper analysis in the separation argument. To this end, we introduce the nonempty convex cone of $\mathbb{R}^{j+1}$
\begin{align*}
    \mathcal{A}_T \triangleq \bigg\{ \Big( D\rho \big( \varphi_0(x^*(T)) \big) \big( \nabla \varphi_0(x^*(T)) \cdot y_{g _1,g_2}(T) \big) , \mathbb{E}\big[ \nabla \varphi_1(x^*(T)) \cdot y_{g_1,g_2}(T) \big] , \dots, & \\
    \mathbb{E}\big[ \nabla \varphi_j(x^*(T)) \cdot y_{g_1,g_2}(T) \big] \Big) \, : \ y_{g_1,g_2}(T) \in \mathcal{R}_T & \bigg\} ,
\end{align*}
and assume by contradiction that 
$$
\mathcal{A}_T \cap (-\infty,0)^{j+1} \neq \emptyset.
$$
The latter identity is tantamount to the existence a progressively measurable selection $t \in [0,T] \mapsto (g_1,g_2)(t) \in T_{F(t,x^*(t))} (f,\sigma)(t,x^*(t),u^*(t))$ such that
\begin{equation}
\label{eq:Contradiction1}
D\rho\big( \varphi_0(x^*(T)) \big) \cdot \big( \nabla \varphi_0(x^*(T)) \cdot y_{g _1,g_2}(T) \big) < 0,
\end{equation}
and 
\begin{equation}
\label{eq:Contradiction2}
\mathbb{E}\big[ \nabla \varphi_i(T,x^*(T)) \cdot y_{g_1,g_2}(T) \big] < 0 \qquad \text{for each $i\in\{1,\dots,j\}$}.
\end{equation}
At this stage, thanks to Theorem \ref{theo:var}, we may find for every $\varepsilon > 0$ a solution $x^{\varepsilon}_{g_1,g_2} \in C^2_{\mathcal{F}}([0,T] \times \Omega,\mathbb{R}^n)$ to \eqref{eq:SDI} which satisfies \eqref{eq:var}. In particular, from \cite[Theorem 8.1.3 and Theorem 8.2.10]{Aubin1990} we readily obtain the existence of a progressively measurable control mapping $u^{\varepsilon}_{g_1,g_2} : [0,T] \times \Omega \rightarrow U$ such that $x^{\varepsilon}_{g_1,g_2} = x_{u^{\varepsilon}_{g_1,g_2}}$ is an admissible trajectory of \eqref{eq:SDE}. Besides, by leveraging Theorem \ref{thm:RiskMeasures}, we may write that
{\small
\begin{equation}
\label{eq:Contradiction3}
\begin{aligned}
    \rho\big( \varphi_0(x^{\varepsilon}_{g_1,g_2}(T)) \big) &  = \rho\big( \varphi_0(x^*(T)) \big) + D\rho\big( \varphi_0(x^*(T)) \big) \cdot \big( \varphi_0(x^{\varepsilon}_{g_1,g_2}(T)) - \varphi_0(x^*(T)) \big) \\
    & \hspace{0.45cm} + o \Big( \| \varphi_0(x^{\varepsilon}_{g_1,g_2}(T)) - \varphi_0(x^*(T)) \|_{L^1_{\mathcal{F}_T}} \Big) \\
    & = \rho\big( \varphi_0(x^*(T)) \big) + D\rho\big( \varphi_0(x^*(T)) \big) \cdot \Big( \varepsilon \nabla\varphi_0(x^*(T)) \cdot y_{g_1,g_2}(T) + o(\varepsilon) \Big) \\
    & \hspace{0.45cm} + o\big( \|x^{\varepsilon}_{g_1,g_2} - x^*\|_{C^1_{\mathcal{F}}} \big) \\
    & \le \rho\big( \varphi_0(x^*(T)) \big) + \varepsilon D\rho\big( \varphi_0(x^*(T)) \big) \cdot \big( \nabla \varphi_0(x^*(T)) \cdot y_{g _1,g_2}(T) \big) + o( \varepsilon )
\end{aligned}
\end{equation}}
where we used hypothesis \ref{hyp:MAC}-$(ii)$ along with the distance estimates of Theorem \ref{theo:var} and the fact that $\partial \rho\big( \varphi_0(x^*(T)) \big) \subset L^{\infty}_{\mathcal{F}_T}(\Omega,\mathbb{R})$ is bounded by Theorem \ref{thm:RiskMeasures}. Analogously, it  holds for every $i \in \{1,\dots,j\}$ that
\begin{equation}
\label{eq:Contradiction4}
\mathbb{E}\big[ \varphi_i(x^{\varepsilon}_{g_1,g_2}(T)) \big] \le \mathbb{E}\big[ \varphi_i(x^*(T)) \big] + \varepsilon \mathbb{E}\big[ \nabla \varphi_i(x^*(T)) \cdot y_{g_1,g_2}(T) \big] + o(\varepsilon) .
\end{equation}
By combining \eqref{eq:Contradiction1}-\eqref{eq:Contradiction3} on the one hand and \eqref{eq:Construction2}-\eqref{eq:Contradiction4} on the other hand, we conclude that whenever $\varepsilon>0$ is small enough, $(x^{\varepsilon}_{g_1,g_2},u^{\varepsilon}_{g_1,g_2})$ is an admissible pair for \eqref{eq:OCP} whose cost is strictly lower than that of $(x^*,u^*)$, which contradicts our standing assumption. Whence, it necessarily holds that $\mathcal{A}_T \cap (-\infty,0)^{j+1} = \emptyset$.

At this stage, by applying yet again the separation theorem, we may infer the existence of a nontrivial multiplier $(\mathfrak{p}_0,\mathfrak{p}) \triangleq (\mathfrak{p}_0,\mathfrak{p}_1,\dots,\mathfrak{p}_{\ell}) \in \mathbb{R}^{j+1} \setminus \{0\}$ such that
\begin{equation*}
\underset{a \in (-\infty,0)^{j+1}}{\sup} (\mathfrak{p}_0,\mathfrak{p}) \cdot a \le \underset{b \in \mathcal{A}_T}{\inf} (\mathfrak{p}_0,\mathfrak{p}) \cdot b.
\end{equation*}
First, we show that we necessarily have $\mathfrak{p}_0 \neq 0$. Indeed, if by contradiction we assume that $\mathfrak{p}_0 = 0$, the latter inequality becomes
$$
\underset{a \in (-\infty,0)^j}{\sup} \mathfrak{p} \cdot a \; \le \; \underset{b \in \mathcal{B}_T}{\inf} \ \mathfrak{p} \cdot b.
$$
Now, since we assumed that there exists at least one element in $c \in (-\infty,0)^j \cap  \mathcal{B}_T$, we may select by continuity another point $a_c \in (-\infty,0)^j$ in such a way that
\begin{equation*}
\mathfrak{p} \cdot a_c \le \sup_{a \in(-\infty,0)^j} \mathfrak{p} \cdot a \; \leq \; \inf_{b \in\mathcal{B}_T} \mathfrak{p} \cdot b \leq \mathfrak{p} \cdot c < \mathfrak{p}\cdot a_c,
\end{equation*}
which leads to a contradiction. Moreover, since $(-\infty,0)^j$ and $\mathcal{B}_T$ are both cones, we further obtain up to a renormalization by $\mathfrak{p}_0$ that $\mathfrak{p}_i \ge 0$ for every $i=1,\dots,j$, and 
\begin{equation} \label{eq:var2}
    D\rho\big( \varphi_0(x^*(T)) \big) \cdot \big( \nabla \varphi_0(x^*(T)) \cdot y_{g _1,g_2}(T) \big) + \sum^j_{i=1} \mathfrak{p}_i \mathbb{E}\Big[ \nabla \varphi_i(x^*(T))\cdot y_{g_1,g_2}(T) \Big] \ge 0 .
\end{equation}
Up to trivially embedding $\mathfrak{p}$ into $\mathbb{R}^{\ell}$, changing its sign and merging \eqref{eq:var2} with \eqref{eq:var1}, there exists a nontrivial multiplier $(\mathfrak{p}_0,\dots,\mathfrak{p}_{\ell}) \in \{-1,0\} \times \mathbb{R}_-^{\ell}$ such that 
$$
(\mathfrak{p}_0,\dots,\mathfrak{p}_{\ell}) \neq 0 \qquad \text{and} \qquad \mathfrak{p}_i \mathbb{E}\big[ \varphi_i(x^*(T)) \big] = 0 ~~ \text{for every $i \in \{1,\dots,\ell\}$},
$$ 
and for which, thanks to Theorem \ref{thm:RiskMeasures}, the following linearized inequality 
\begin{equation} 
\label{eq:varWeak}
\underset{\xi \in \partial \rho( \varphi_0(x^*(T)))}{\inf} \ \mathbb{E} \Bigg[ \bigg( \xi \mathfrak{p}_0 \nabla \varphi_0(x^*(T)) + \sum^{\ell}_{i=1} \mathfrak{p}_i \nabla \varphi_i(x^*(T)) \bigg)\cdot y_{g_1,g_2}(T) \Bigg] \le 0, 
\end{equation}
holds for any selection $t \in[0,T] \mapsto (g_1(t),g_2(t)) \in T_{F(t,x^*(t))} (f,\sigma )\big(t,x^*(t),u^*(t) \big)$. In particular, the Lagrange multiplier $(\mathfrak{p}_0,\dots,\mathfrak{p}_{\ell})$ is non-trivial, and complies with the complementary slackness conditions \eqref{eq:PMPSlack} of the PMP. 

\vspace{0.25cm}


\noindent \textbf{Step 4 -- Universal separation theorem.} In what follows, we extract further information from \eqref{eq:varWeak}, by observing that the latter inequality can be rewritten as 
\begin{equation}
\label{eq:varWeakBis}
\sup_{(g_1,g_2)} \underset{\xi \in \partial \rho( \varphi_0(x^*(T)))}{\inf} \ \mathbb{E}\left[ \bigg( \xi \mathfrak{p}_0 \nabla \varphi_0(x^*(T)) + \sum^{\ell}_{i=1} \mathfrak{p}_i \nabla \varphi_i(x^*(T)) \bigg)\cdot y_{g_1,g_2}(T) \right] \le 0,
\end{equation}
which leads us to consider the mapping
$$
\mathcal{H}(\xi,(g_1,g_2)) \triangleq \mathbb{E}\left[ \bigg( \xi \mathfrak{p}_0 \nabla \varphi_0(x^*(T)) + \sum^{\ell}_{i=1} \mathfrak{p}_i \nabla \varphi_i(x^*(T)) \bigg)\cdot y_{g_1,g_2}(T) \right],
$$
that is defined for each $\xi \in L^{\infty}_{\mathcal{F}_T}(\Omega,\mathbb{R})$ and every progressively measurable selection $t \in[0,T] \mapsto (g_1(t),g_2(t)) \in T_{F(t,x^*(t))} \big( (f,\sigma)(t,x^*(t),u^*(t)) \big)$. 

By Theorem \ref{thm:RiskMeasures}, the set $\partial \rho\big( \varphi_0(x^*(T)) \big) \subset L^{\infty}_{\mathcal{F}_T}(\Omega,\mathbb{R})$ is convex and weakly-$^*$ compact, whereas the set of all progressively measurable selections
$$
t \in[0,T] \mapsto (g_1(t),g_2(t)) \in T_{F(t,x^*(t))}((f,\sigma)(t,x^*(t),u^*(t))
$$
is a convex subset of $L^2_{\mathcal{F}}([0,T] \times\Omega,\mathbb{R}^{n + n \times d})$. Moreover, it can be checked that
\begin{equation*}
\xi \mapsto \mathcal{H}(\xi,(g_1,g_2))
\end{equation*}
is continuous for the weak-$^*$ topology of $L^{\infty}_{\mathcal{F}_T}(\Omega,\mathbb{R})$ since $\mathfrak{p}_0 \nabla \varphi_0(x^*(T)) \cdot y_{g_1,g_2}(T) \in L^1_{\mathcal{F}_T}(\Omega,\mathbb{R}^n)$. On the other hand, it follows from Lemma \ref{lemma:contY} that
$$
(g_1,g_2) \mapsto \mathcal{H}(\xi,(g_1,g_2))
$$ 
is continuous for the strong topology of $L^2_{\mathcal{F}}([0,T]\times \Omega,\mathbb{R}^{n + n \times d})$. Since both topologies under consideration are Hausdorff (see e.g. \cite[Proposition 3.11]{Brezis} for the former), it follows from the separation result of Theorem \ref{thm:Sion} that we can rewrite \eqref{eq:varWeakBis} as 
\begin{equation*}
\inf_{\xi \in \partial \rho( \varphi_0(x^*(T)))} \sup_{(g_1,g_2)} \mathbb{E} \Bigg[ \bigg( \xi \mathfrak{p}_0 \nabla \varphi_0(x^*(T)) + \sum^{\ell}_{i=1} \mathfrak{p}_i \nabla \varphi_i(x^*(T)) \bigg)\cdot y_{g_1,g_2}(T) \Bigg] \le 0.
\end{equation*}
Because the supremum of a family of lower semicontinuous functions remains lower semicontinuous for that same topology (see e.g. \cite[Proposition 3.2.3]{Attouch2005}), the mapping 
$$
\xi \mapsto \sup_{(g_1,g_2)} \mathcal{H}(\xi,(g_1,g_2))
$$ 
is weakly-$^*$ lower-semicontinuous. This, along with the fact that $\partial \rho\big( \varphi_0(x^*(T)) \big) \subset L^{\infty}_{\mathcal{F}_T}(\Omega,\mathbb{R})$ is weakly-$^*$ compact, yields the existence of $\xi^* \in \partial \rho\big( \varphi_0(x^*(T)) \big)$ such that
\begin{equation}
\label{eq:uniVar0}
\sup_{(g_1,g_2)} \mathcal{H}(\xi^*,(g_1,g_2)) = \min_{\xi \in \partial \rho( \varphi_0(x^*(T)))}  \sup_{(g_1,g_2)} \ \mathcal{H}(\xi,(g_1,g_2)) \le 0.
\end{equation}
In particular, this directly provides us with the condition \eqref{eq:PMPMaximisationOther} of the PMP as consequence of the characterization of $\partial \rho\big( \varphi_0(x^*(T)) \big)$ given in Theorem \ref{thm:RiskMeasures}.

\vspace{5pt}


\noindent \textbf{Step 5 -- Costate dynamics and maximisation condition.} In what follows, we derive the adjoint equation and recover the maximality condition from \eqref{eq:uniVar0}, which will conclude the proof of Theorem \ref{thm:controlledDiff}. While the underlying computations come from classical BSDE theory, see e.g. \cite{Frankowska2020}, we reproduce them below for the sake of readability and completeness. Notice first that \eqref{eq:uniVar0} straightforwardly implies that
\begin{equation} 
\label{eq:uniVar}
\mathbb{E} \Bigg[ \bigg( \xi^* \mathfrak{p}_0 \nabla \varphi_0(x^*(T)) + \sum^{\ell}_{i=1} \mathfrak{p}_i \nabla \varphi_i(x^*(T)) \bigg)\cdot y_{g_1,g_2}(T) \Bigg] \le 0
\end{equation}
for every selection $t \in [0,T] \mapsto (g_1,g_2)(t) \in T_{F(t,x^*(t))} \big( (f,\sigma)(t,x^*(t),u^*(t)) \big)$. We denote by $\phi , \psi \in L^2_{\mathcal{F}}([0,T]\times\Omega,\mathbb{R}^{n \times n})$ the unique (up to stochastic indistinguishability) solutions of the matrix-valued stochastic differential equations 
\begin{equation*}
\phi(t) = \textnormal{Id} + \int_0^t A(s) \phi(s) \mathrm{d}s + \sum^d_{i=1} \int_0^t D_i(s) \phi(s) \mathrm{d}W^i_s 
\end{equation*}
and
\begin{equation*}
\psi(t) = \textnormal{Id} - \int_0^t \psi(s) \bigg( A(s) - \sum^d_{i=1} D_i^2(s) \bigg) \mathrm{d}s - \sum^d_{i=1} \int_0^t \psi(s) D_i(s) \mathrm{d}W^i_s, 
\end{equation*}
whose well-posedness are guaranteed e.g. by \cite[Section 1.6.3]{Yong1999}. We list in the following lemma some properties of these maps, whose proofs rely on simple componentwise applications of the It\^o formula in the spirit e.g. of \cite[Theorem 6.14, Chapter 1]{Yong1999}.

\begin{lem}
The maps $\phi , \psi$ are elements of $C^{\beta}_{\mathcal{F}}([0,T] \times \Omega,\mathbb{R}^{n \times n})$ for every $\beta \in [2,+\infty)$, and satisfy the identity $\psi(t) = \phi(t)^{-1}$ for all times $t \in[0,T]$.
\end{lem}

Thanks to \cite[Theorem 6.14, Chapter 1]{Yong1999}, any solution $y_{g_1,g_2} \in C^2_{\mathcal{F}}([0,T]\times\Omega,\mathbb{R}^n)$ of \eqref{eq:LSDE} can be expressed as
\begin{equation}
\label{eq:LinearisedRepresentation}
\begin{aligned}
y_{g_1,g_2}(t) = \phi(t) \int^t_0 \psi(s) \bigg( g_1(s) - \sum^d_{i=1} D_i(s) g_2^i(s) \bigg) \; \mathrm{d}s + \phi(t) \sum^d_{i=1} \int^t_0 \psi(s) g_2^i(s) \mathrm{d}W^i_s,
\end{aligned}
\end{equation}
for all times $t \in[0,T]$. At this stage, let it be noted that the stochastic process
\begin{equation*}
t \in [0,T] \mapsto \mathbb{E} \Bigg[ \phi(T)^{\top} \bigg( \xi^* \mathfrak{p}_0 \nabla \varphi_0(x^*(T)) + \sum^{\ell}_{i=1} \mathfrak{p}_i \nabla \varphi_i(x^*(T)) \bigg) \bigg| \, \mathcal{F}_t \Bigg] 
\end{equation*}
is a martingale that is uniformly bounded in $L^2$ as a direct consequence of Jensen's and H\"older's inequalities. Therefore, thanks to Theorem \ref{thm:Mart}, there exist a vector $N \in \mathbb{R}^n$ and a process $\mu \in L^2_{\mathcal{F}}([0,T] \times \Omega,\mathbb{R}^{n \times d})$ such that
\begin{equation}
\label{eq:NChiDef}
\begin{aligned}
    N + \sum^d_{j=1} \chi_j(t) &\triangleq N + \sum^d_{j=1} \int^t_0 \mu_j(s) \; \mathrm{d}W^j_s \\
    &= \mathbb{E}\Bigg[ \phi(T)^{\top} \bigg( \xi^* \mathfrak{p}_0 \nabla \varphi_0(x^*(T)) + \sum^{\ell}_{i=1} \mathfrak{p}_i \nabla \varphi_i(x^*(T)) \bigg) \bigg| \, \mathcal{F}_t \Bigg],
\end{aligned}
\end{equation}
for every $t \in [0,T]$. It then follows from \eqref{eq:LinearisedRepresentation} and \eqref{eq:NChiDef} that 
\begin{align*}
    &\mathbb{E} \Bigg[ \bigg( \xi^* \mathfrak{p}_0 \nabla \varphi_0(x^*(T)) + \sum^{\ell}_{i=1} \mathfrak{p}_i \nabla \varphi_i(x^*(T)) \bigg)\cdot y_{g_1,g_2}(T) \Bigg] \\
    &= \mathbb{E} \Bigg[ \mathbb{E} \bigg[ \phi(T)^{\top} \bigg( \xi^* \mathfrak{p}_0 \nabla \varphi_0(x^*(T)) + \sum^{\ell}_{i=1} \mathfrak{p}_i \nabla \varphi_i(x^*(T)) \bigg) \bigg| \, \mathcal{F}_T \bigg] \cdot \\
    &\qquad \qquad  \bigg( \int^T_0 \psi(s) \bigg( g_1(s) - \sum^d_{i=1} D_i(s) g_2^i(s) \bigg) \mathrm{d}s + \sum^d_{i=1} \int^T_0 \psi(s) g_2^i(s)  \mathrm{d}W^i_s \bigg) \Bigg] \\
    &= \mathbb{E}\Bigg[ \int^T_0 N \cdot \psi(s) \bigg( g_1(s) - \sum^d_{i=1} D_i(s) g_2^i(s) \bigg) \mathrm{d}s \Bigg]\\
    & \qquad + \mathbb{E} \Bigg[ \sum^d_{j=1} \bigg( \int^T_0 \mu_j(s) \mathrm{d}W^j_s \bigg) \cdot \bigg( \int^T_0 \psi(s) \bigg( g_1(s) - \sum^d_{i=1} D_i(s) g_2^i(s) \bigg) \mathrm{d}s \bigg) \Bigg] \\
    &\qquad + \mathbb{E}\Bigg[\sum^d_{i,j=1} \bigg( \int^T_0 \mu_j(s) \mathrm{d}W^j_s \bigg) \cdot \bigg( \int^T_0 \psi(s) g_2^i(s) \mathrm{d}W^i_s \bigg) \Bigg].
\end{align*}
At this stage, thanks to integration by parts formula of the It\^o calculus (see e.g. \cite[p. 116]{LeGall2016}), it further holds that
\begin{align*}
    \mathbb{E} \Bigg[ \sum^d_{j=1} \bigg( \int^T_0 \mu_j(s) \mathrm{d}W^j_s \bigg) \cdot & \bigg( \int^T_0 \psi(s) \bigg( g_1(s) - \sum^d_{i=1} D_i(s) g_2^i(s) \bigg) \mathrm{d}s \bigg) \Bigg] \\
    &= \mathbb{E} \Bigg[ \sum^d_{j=1} \int^T_0 \chi_j(s) \cdot \psi(s) \bigg( g_1(s) - \sum^d_{i=1} D_i(s) g_2^i(s) \bigg) \mathrm{d}s \Bigg] ,
\end{align*}
as well as
\begin{align*}
    & \mathbb{E}\Bigg[\sum^d_{i,j=1} \bigg( \int^T_0 \mu_j(s) \mathrm{d}W^j_s \bigg) \cdot \bigg( \int^T_0 \psi(s) g_2^i(s) \mathrm{d}W^i_s \bigg) \Bigg] \\
    & = \mathbb{E} \bigg[ \sum^d_{i,j=1} \int_{0}^T \mu_j(s) \cdot \psi(s)g^i_2(s) \mathrm{d} \langle W^i,W^j \rangle_s \bigg] = \mathbb{E} \bigg[ \sum^d_{i,j=1} \int_{0}^T \mu_j(s) \cdot \psi(s)g^i_2(s) \mathrm{d} s \bigg]
\end{align*}
by \cite[Section 4.3 and Section 5.1 Formula (5.7)]{LeGall2016}, wherein $\langle \cdot , \cdot \rangle$ stands for standard the quadratic variation of a continuous martingale (see e.g. \cite[Section 4.3]{LeGall2016}). Merging the previous computations finally leads to
\begin{align*}
    &\mathbb{E} \Bigg[ \bigg( \xi^* \mathfrak{p}_0 \nabla \varphi_0(x^*(T)) + \sum^{\ell}_{i=1} \mathfrak{p}_i \nabla \varphi_i(x^*(T)) \bigg)\cdot y_{g_1,g_2}(T) \Bigg] \\
    &= \mathbb{E} \Bigg[ \int^T_0 \bigg( N + \sum^d_{j=1} \chi_j(s) \bigg) \cdot \psi(s) g_1(s) \mathrm{d}s \Bigg] \\
    &\qquad + \mathbb{E} \Bigg[ \int^T_0 \sum^d_{i=1} \Bigg( \mu_i(s) \cdot \psi(s) g_2^i(s) - \bigg( N + \sum^d_{j=1} \chi_j(s) \bigg) \cdot \psi(s) D_i(s) g_2^i(s) \bigg)  \mathrm{d}s \Bigg].
\end{align*}
Notice at this point that, by defining the costate curves
\begin{equation}
\label{eq:CostateDef}
\left\{
\begin{aligned}
    p^*(t) & \triangleq \psi(t)^{\top} \bigg( N + \sum^d_{j=1} \chi_j(t) \bigg), \\
    q^*(t) & \triangleq \bigg[ \Big( \psi(t)^{\top} \mu_1(t)  - D_1(t)^{\top} p^*(t) \Big) \Big| \dots \Big| \Big( \psi(t)^{\top} \mu_d(t)  - D_d(t)^{\top} p^*(t) \Big) \bigg],
\end{aligned}
\right.
\end{equation}
for almost every $t\in[0,T]$, the variational inequality \eqref{eq:uniVar} can be rewritten as
\begin{equation}
\label{eq:VariationalInterm}
\begin{aligned}
    & \mathbb{E} \bigg[ \int^T_0 p^*(t) \cdot g_1(s) \mathrm{d}s + \int^T_0 \sum^d_{i=1} q_i^*(t) \cdot g_2(t) \mathrm{d}s \bigg] \leq 0
\end{aligned}
\end{equation}
for every selection $t \in [0,T] \mapsto (g_1(t),g_2(t)) \in T_{F(t,x^*(t))} \big( (f,\sigma)(t,x^*(t),u^*(t)) \big)$.  

We are now going to show that \eqref{eq:PointwisePMPIneq} in fact yields the maximization condition. For any $u \in \mathcal{U}$, observe that the maps defined by 
\begin{equation*}
\hspace{3cm} \left\{
\begin{aligned}
& g^u_1(t) \triangleq f(t,x^*(t),u(t)) - f(t,x^*(t),u^*(t)) , \\
& g^u_2(t) \triangleq \sigma(t,x^*(t),u(t)) - \sigma(t,x^*(t),u^*(t)), 
\end{aligned}
\right.
\end{equation*}
for almost every $t \in[0,T]$ are such that 
\begin{equation*}
(g^u_1,g^u_2)(t) \in T_{F(t,x^*(t))} \big( (f,\sigma)(t,x^*(t),u^*(t)) \big)
\end{equation*}
by construction, since we assumed that the sets $F(t,x^*(t)) \subset \mathbb{R}^{n+d \times n}$ are convex. This together with \eqref{eq:VariationalInterm} and the definition \eqref{eq:Hamiltonian} of the Hamiltonian implies that
\begin{equation}
\label{eq:PointwisePMPIneq}
\mathbb{E} \Bigg[ \int^T_0 \Big( H(s,x^*(s),p^*(s),q^*(s),u(s)) - H(s,x^*(s),p^*(s),q^*(s),u^*(s)) \Big) \mathrm{d}s \Bigg] \leq 0,
\end{equation}
for every $u \in \mathcal{U}$. Given an integer $m \geq 1$, consider the closed subset of control values
\begin{equation}
\label{eq:TildeUDef}
\begin{aligned}
\tilde{U}_m(t,\omega) := \bigg\{ u \in U \, : \,\; & H(t,\omega,u,x^*(t,\omega),p^*(t,\omega),q^*(t,\omega)) \geq \\
& H(t,\omega,x^*(t,\omega),u^*(t,\omega),p^*(t,\omega),q^*(t,\omega)) + \tfrac{1}{m}   \bigg\}, 
\end{aligned}
\end{equation}
and suppose by contradiction that the corresponding set
\begin{equation}
\label{eq:TildeFDef}
\tilde{\mathcal{F}}_m := \Big\{ (t,\omega) \in [0,T] \times \Omega \, : \, \tilde{U}_m(t,\omega) \neq\emptyset \Big\} \subset [0,T] \times \Omega, 
\end{equation}
which, by construction, is measurable with respect to the progressive $\sigma$-algebra generated by $\mathcal{F}$, has positive measure. Then, by choosing any admissible control signal $\tilde{u}_m : [0,T]\times\Omega \to U$ such that $\tilde{u}_m(t,\omega) \in \tilde{U}_m(t,\omega)$ for almost every $(t,\omega) \in\tilde{\mathcal{F}}_m$ and $\tilde{u}_m(t,\omega) = u^*(t,\omega)$ otherwise, it holds that 
{\small
\begin{equation*}
\begin{aligned}
\mathbb{E} \bigg[ \int^T_0 \hspace{-0.1cm} \Big( H(s,x^*(s),p^*(s),q^*(s),u(s)) - H(s,x^*(s),p^*(s),q^*(s),u^*(s)) \Big) \mathrm{d}s \bigg] \hspace{-0.075cm} \geq \hspace{-0.075cm} \tfrac{1}{m} (\mathrm{d}t \otimes \mathbb{P})(\tilde{\mathcal{F}}_m),
\end{aligned}
\end{equation*}}
which contradicts \eqref{eq:PointwisePMPIneq}. Whence, the set defined by 
\begin{equation*}
\mathcal{\tilde{F}}_{\infty} := \bigcup_{m \geq 1} \mathcal{\tilde{F}}_m
\end{equation*}
necessarily has zero $\mathrm{d}t \otimes \mathbb{P}$-measure, which together with \eqref{eq:TildeUDef}-\eqref{eq:TildeFDef} implies that the maximisation condition \eqref{eq:PMPMaximisation} of the PMP holds. 

To conclude, we now shift our focus to the dynamics of the costate variable. First, note that $p^*$ is adapted to the filtration $\mathcal{F}$ by construction, and that it has continuous sample-paths. Moreover, we may infer from a straightforward use of Doob's, Jensen's, and H\"older's inequalities, along with the facts that $\phi , \psi \in C^{\beta}([0,T] \times \Omega,\mathbb{R}^{n \times n})$ for every $\beta \in [2,\infty)$ and $\xi^* \in L^{\infty}_{\mathcal{F}_T}(\Omega,\mathbb{R})$, that $p \in C^2_{\mathcal{F}}([0,T] \times \Omega,\mathbb{R}^n)$. In addition, we have by It\^o's formula that
\begin{align*}
    \psi(t)^{\top} \bigg( \sum^d_{j=1} \int^t_0 \mu_j(s) \mathrm{d}W^j_s \bigg) & = \sum^d_{j=1} \int^t_0 \psi(s)^{\top} \mu_j(s) \mathrm{d}W^j_s - \sum^d_{i=1} \int^t_0 D_i(s)^{\top} \psi(s)^{\top} \mu_i(s) \mathrm{d}s \\
    & \quad - \sum^d_{i=1} \int^t_0 D_i(s)^{\top} \psi(s)^{\top} \bigg( \sum^d_{j=1} \chi_j(s) \bigg) \mathrm{d}W^i_s \\
    & \quad - \int^t_0 \bigg( A(s)^{\top} - \sum^d_{i=1} (D_i(s)^2)^{\top} \bigg) \psi(s)^{\top} \bigg( \sum^d_{j=1} \chi_j(s) \bigg) \mathrm{d}s
\end{align*}
for all times $t \in [0,T]$. This, combined with the definition \eqref{eq:CostateDef} of $(p^*,q^*)$ along with that of the Hamiltonian in \eqref{eq:Hamiltonian} allows us to deduce that
\begin{equation}
\label{eq:IntermediateCostate}
p^*(t) = N - \int^t_0 \frac{\partial H}{\partial x} \big( s,x^*(s),u^*(s),p^*(s),q^*(s) \big) \mathrm{d}s + \int^t_0 q^*(s) \mathrm{d}W_s,
\end{equation}
where we also used the fact that $\psi(0) = \textnormal{Id}$ by construction. Regarding the terminal condition, observe that owing to \eqref{eq:NChiDef} along with \eqref{eq:CostateDef}, there holds 
\begin{equation*}
\begin{aligned}
p^*(T) & = \psi(T)^{\top} \mathbb{E}\Bigg[ \phi(T)^{\top} \bigg( \xi^* \mathfrak{p}_0 \nabla \varphi_0(x^*(T)) + \sum^{\ell}_{i=1} \mathfrak{p}_i \nabla \varphi_i(x^*(T)) \bigg) \bigg| \, \mathcal{F}_T \Bigg] \\
& = \xi^* \mathfrak{p}_0 \nabla \varphi_0(x^*(T)) + \sum^{\ell}_{i=1} \mathfrak{p}_i \nabla \varphi_i(x^*(T))
\end{aligned}
\end{equation*}
because $\phi(T) = \psi(T)^{-1}$ and the random variable in the conditional expectation is $\mathcal{F}_T$-measurable, we precisely recover the adjoint dynamics posited in \eqref{eq:PMPAdjoint} of Theorem \ref{thm:controlledDiff}. Finally, by repeating the argument developed e.g. in the proof of \cite[Theorem 2.2, Section 7.2]{Yong1999}, we obtain that $q^* \in L^2_{\mathcal{F}}([0,T] \times \Omega,\mathbb{R}^{n \times d})$.


\subsection{Uncontrolled Diffusion}
\label{sec:uncontrolledDiff}

We now turn our attention towards the simpler scenario in which the control variable does not appear in the diffusion, namely $\sigma(t,\omega,x,u) \equiv \sigma(t,\omega,x)$. Unlike the previous situation, we may relax our assumptions and obtain the PMP without hypothesis \ref{hyp:ACD}.

\begin{thm}[Risk-averse PMP for \eqref{eq:OCP} with uncontrolled diffusion] \label{Theo:uncontrolledDiff}
Suppose that the diffusion term is independent of the control variable, that hypotheses \textnormal{\ref{hyp:MAS}} and \textnormal{\ref{hyp:MAC}} are satisfied, and let $(x^*,u^*)$ be a local minimum for \textnormal{\eqref{eq:OCP}}. Then, the conclusions of Theorem \ref{thm:controlledDiff} hold.
\end{thm}

The proof of Theorem \ref{Theo:uncontrolledDiff} is almost identical to that of Theorem \ref{thm:controlledDiff}, and we shall thus only highlight the few key modifications needed with respect to the argument developed in Section \ref{sec:controlledDiff}. In this context, we will work with the set-valued map
\begin{equation*}
    F : (t,\omega,x) \in[0,T] \times \Omega \times \mathbb{R}^n \rightrightarrows \Big\{ f(t,\omega,x,u) \in \mathbb{R}^n : \ u \in U \Big\} \subset \mathbb{R}^n.
\end{equation*}
Adopting the convention introduced in Remark \ref{thm:RiskMeasures}, the set-valued maps $(t,\omega,x) \in[0,T]\times\Omega\times\mathbb{R}^n \rightrightarrows F(t,x)$ and $(t,\omega,x) \in[0,T]\times\Omega\times\mathbb{R}^n \rightrightarrows \overline{\textnormal{co}} F(t,x)$ have nonempty compact images, are integrably bounded and progressively measurable-Lipschitz under hypotheses \ref{hyp:MAS}. Thus by Theorem \ref{thm:Selections}, the progressive multifunction $t \in[0,T] \rightrightarrows T_{\overline{\textnormal{co}} F(t,x^*(t))}\big( f(t,x^*(t),u^*(t)) \big)$ admits progressively measurable selections 
\begin{equation*}
t \in[0,T]\mapsto g(t) \in T_{\overline{\textnormal{co}} F(t,x^*(t))} f(t,x^*(t),u^*(t)).     
\end{equation*}
In what follows given such a selection, we denote by $y_{g} \in C^2_{\mathcal{F}}([0,T] \times \Omega,\mathbb{R}^n)$ the unique (up to stochastic indistinguishability) solution of the stochastic differential equation
\begin{equation}
\label{eq:LSDEbis}
\tag{$\textnormal{LSDE}_{g}$}
\left\{
\begin{aligned}
\mathrm{d}y(t) & = \Big( A(t) y(t) + g(t) \Big) + \sum^d_{i=1} D_i(t) y(t) \mathrm{d}W^i_t , \\
y(0) & = 0,
\end{aligned}
\right.
\end{equation}
where we used the condensed notations
$$
A(t) \triangleq \frac{\partial f}{\partial x}(t,x^*(t),u^*(t)) \qquad \text{and} \qquad D_i(t) \triangleq \frac{\partial \sigma_i}{\partial x}(t,x^*(t)),
$$
for almost every $t\in[0,T]$ and each $i \in\{ 1,\dots,d\}$. Thanks to the relaxation property of Theorem \ref{theo:Relaxation}, the variational linearization studied in Theorem \ref{theo:var} can be adapted and improved as follows for stochastic dynamics with uncontrolled diffusions.

\begin{thm}[Variational linearization for uncontrolled diffusions] 
\label{thm:varUncontrolled}
For any progressively measurable selection $t \in [0,T] \mapsto g(t) \in T_{\overline{\textnormal{co}} F(t,x^*(t))} f(t,x(t),u^*(t))$ and each $\varepsilon > 0$, there exists a solution $x^{\varepsilon}_{g} \in C^2_{\mathcal{F}}([0,T] \times \Omega,\mathbb{R}^n)$ to \eqref{eq:SDI'} such that
\begin{equation*}
\underset{\varepsilon \rightarrow 0^+}{\lim} \frac{1}{\varepsilon} \mathbb{E} \bigg[ \, \underset{t \in [0,T]}{\sup} \big\| x^{\varepsilon}_{g}(t) - x^*(t) - \varepsilon y_g(t) \big\| \bigg] = 0 ,
\end{equation*}
where $y_{g} \in C^2_{\mathcal{F}}([0,T] \times \Omega,\mathbb{R}^n)$ is the unique solution of \eqref{eq:LSDEbis}.
\end{thm}

\begin{proof}
By repeating the argument outlined earlier in the proof of Theorem \ref{theo:var}, one may readily check that there exists a solution $\bar x^{\varepsilon}_{g} \in C^2_{\mathcal{F}}([0,T] \times \Omega,\mathbb{R}^n)$ to the stochastic differential inclusion
\begin{equation*}
\left\{
\begin{aligned}
\mathrm{d}x(t) & \in \overline{\textnormal{co}} F(t,x(t))  \mathrm{d}t + \sigma(t,x(t)) dW_t, \\
x(0) & = x_0,
\end{aligned}
\right.
\end{equation*}
which satisfies 
\begin{equation*}
\underset{\varepsilon \rightarrow 0^+}{\lim} \ \frac{1}{\varepsilon} \mathbb{E}\bigg[\, \underset{t \in [0,T]}{\sup} \| \bar x^{\varepsilon}_{g}(t) - x^*(t) - \varepsilon y_{g}(t) \| \bigg] = 0.
\end{equation*}
Besides by Theorem \ref{theo:Relaxation}, there exists a solution $x^{\varepsilon}_{g} \in C^2_{\mathcal{F}}([0,T] \times \Omega,\mathbb{R}^n)$ to \eqref{eq:SDI'} that is such that
$$
\underset{\varepsilon \rightarrow 0^+}{\lim} \ \frac{1}{\varepsilon^2} \ \mathbb{E} \bigg[ \, \underset{t \in [0,T]}{\sup} \| x^{\varepsilon}_{g}(t) - \bar x^{\varepsilon}_{g}(t) \|^2 \bigg] = 0 ,
$$
from whence the thesis follows.
\end{proof}

By repeating the arguments of Step 1 and Step 2 of Section \ref{sec:controlledDiff} while using the variational linearization of Theorem \ref{thm:varUncontrolled} instead of Theorem \ref{theo:var}, one can again recover the existence of Lagrange multipliers $(\mathfrak{p}_0,\dots,\mathfrak{p}_{\ell}) \in \{0,-1\} \times \mathbb{R}_-^{\ell}$ satisfying
\begin{equation*}
(\mathfrak{p}_0,\dots,\mathfrak{p}_{\ell}) \neq 0 \qquad \text{and} \qquad \mathfrak{p}_i \mathbb{E}\big[ \varphi_i(x^*(T)) \big] = 0 ~~ \text{for every $i \in \{1,\dots,\ell\}$},
\end{equation*} 
such that the variational inequality 
\begin{equation} 
\label{eq:varWeakUncontrolled}
\underset{\xi \in \partial \rho (\varphi_0(x^*(T)))}{\inf} \mathbb{E} \Bigg[ \bigg( \xi \mathfrak{p}_0 \nabla \varphi_0(x^*(T)) + \sum^{\ell}_{i=1} \mathfrak{p}_i \nabla \varphi_i(x^*(T)) \bigg) \cdot y_{g_1}(T) \Bigg] \le 0,
\end{equation}
holds for any progressively measurable selection
$$t \in [0,T] \mapsto g(t) \in T_{\overline{\textnormal{co}} F(t,x^*(t))} f(t,x^*(t),u^*(t)) .
$$
From there on, one can prove the PMP by repeating verbatim the arguments elaborated in Step 3, Step 4 and Step 5 of Section \ref{sec:controlledDiff}, thus details are skipped.


\section{Examples of application}
\label{sec:Example}

\setcounter{equation}{0} \renewcommand{\theequation}{\thesection.\arabic{equation}}

In this section, we briefly discuss general examples of risk functions and risk-averse stochastic optimal control problems which are encompassed by our results. In this context, we will consider the simple case in which $(x^*,u^*)$ is  a local minimum for \textnormal{\eqref{eq:OCP}} in the case where there is no control in the diffusion and no final-time constraints. Then, Theorem \ref{Theo:uncontrolledDiff} shall provide us with the existence of stochastic processes $(p^*,q^*) \in C^2_{\mathcal{F}}([0,T] \times \Omega,\mathbb{R}^n) \times L^2_{\mathcal{F}}([0,T] \times \Omega,\mathbb{R}^{d \times n})$ and a risk parameter $\xi^* \in \partial \rho \big( \varphi_0(x^*(T)) \big)$ for which \eqref{eq:PMPMaximisationOther}, \eqref{eq:PMPAdjoint}, and \eqref{eq:PMPMaximisation} hold with $(\mathfrak{p}_0,\dots,\mathfrak{p}_{\ell}) = (-1,0,\dots,0)$.


\subsection{Examples of risk-parameters characterization}

Suppose at first that $\rho : L^1_{\mathcal{F}_T}(\Omega,\mathbb{R}) \rightarrow \mathbb{R}$ is Fr\'echet differentiable, as it was for instance assumed in \cite{Isohatala2021}. This situation includes for instance the $\log$-$\exp$ utility function and the mean-variance risk measures, see e.g. \cite{Shapiro2021}. In that case, $\partial \rho(Z) = \{ \nabla \rho(Z) \}$ for every $Z \in L^1_{\mathcal{F}_T}(\Omega,\mathbb{R})$, and the result of Theorem \ref{thm:controlledDiff} hold with the uniquely determined risk parameter
\begin{equation*}
\xi^* = \nabla \rho\big( \varphi_0(x^*(T)) \big) . 
\end{equation*}
Suppose now that $\rho: L^1_{\mathcal{F}_T}(\Omega,\mathbb{R}) \rightarrow \mathbb{R}$ is the prototypical example of subdifferentiable risk measure given by the \textit{Average-Value-at-Risk} of a random variable $Z \in L^1_{\mathcal{F}_T}(\Omega,\mathbb{R})$ with level $\alpha \in (0,1]$, namely
\begin{equation}
\label{eq:AvarDef}
\rho_{\alpha}(Z) = \textnormal{AV@R}_{\alpha}(Z) \triangleq \underset{t \in \mathbb{R}}{\inf} \ \left( t + \frac{1}{\alpha} \mathbb{E}\big[ \max ( Z - t , 0 ) \big] \right).
\end{equation}
In that case, the results of Theorem \ref{thm:controlledDiff} hold for some $\xi^* \in \partial \rho\big( \varphi_0(x^*(T)) \big)$, which satisfies in particular \eqref{eq:PMPMaximisationOther}. From \cite[Example 6.16]{Shapiro2021}, there exists a $(1-\alpha)$-quantile
\begin{equation*}
\begin{aligned}
Q_{1-\alpha}(\varphi_0(x^*(T))) \in \bigg[ \, & \inf \Big\{t \in \mathbb{R} ~:~  H_{\varphi_0(x^*(T))}(t) \geq (1-\alpha) \Big\} , \\
& \sup \Big\{t \in \mathbb{R} ~:~  H_{\varphi_0(x^*(T))}(t) \leq (1-\alpha) \Big\} \, \bigg]
\end{aligned}
\end{equation*}
of the cumulative distribution function $H_{\varphi_0(x^*(T))} : \mathbb{R} \to [0,1]$ of $\varphi_0(x^*(T))$ such that
\begin{equation} \label{eq:RiskParamAVaR}
\xi^*(\omega) = \left\{
\begin{aligned}
0 & \quad \textnormal{if} \quad \varphi_0(x^*(T,\omega)) < Q_{1-\alpha}(\varphi_0(x^*(T))), \\
\lambda^* \in \left( 0 , \tfrac{1}{\alpha} \right) & \quad \textnormal{if} \quad \varphi_0(x^*(T,\omega)) = Q_{1-\alpha}(\varphi_0(x^*(T))), \\
\tfrac{1}{\alpha} & \quad \textnormal{if} \quad \varphi_0(x^*(T,\omega)) > Q_{1-\alpha}(\varphi_0(x^*(T))), \\
\end{aligned}
\right. \quad \text{with} ~~ \mathbb{E}[ \xi^* ] = 1.
\end{equation}
%

\subsection{The risk-averse double integrator problem}

In addition to the computational examples provided hereinabove, we discuss the application of the PMP of Theorem \ref{Theo:uncontrolledDiff} to the following \textit{stochastic optimal planning} problem
\begin{equation*}
\label{eq:SOP}
\tag{$\textnormal{SOP}$}
\left\{
\begin{aligned}
\underset{u \in \mathcal{U}}{\min} \ & \textnormal{AV@R$_{\alpha}$} \big( \tfrac{1}{2} |y(T)-y_T|^2\big), \\
\textnormal{s.t.} ~ & \left\{
\begin{aligned}
&\mathrm{d} y(t) = v(t) \mathrm{d}t + \mathrm{d} W_t, ~~ & y(0) = y_0,  \\
& \mathrm{d} v(t) = u(t) \mathrm{d}t, ~~ & v(0) = v_0,  
\end{aligned}
\right.
\end{aligned}
\right.
\end{equation*}
in which $y_0,v_0,y_T \in \mathbb{R}$ are given such that $y_0 < y_T$, the control set is defined by $\mathcal{U} := L^2([0,T],[-1,1]$), and the average value-at-risk is defined as in \eqref{eq:AvarDef}. 

In what follows, we show that the PMP of Theorem \ref{Theo:uncontrolledDiff} provides a necessary condition for optimal solutions of \eqref{eq:SOP} to be \textit{safe}, in the sense 
\begin{equation}
\label{eq:safe}
\textnormal{AV@R$_{\alpha}$}(y(T)) < y_T. 
\end{equation}
Our definition of safe optimal solutions to \eqref{eq:SOP} is driven by the applications, and the rationale behind it is the following. Imagine for instance that \eqref{eq:SOP} models a one-dimensional traffic lane over which one aims at steering a vehicle from some station $y_0$ to a point which lies as close as possible to the end of the lane $y_T$. It is then of paramount importance that the vehicle stops with high probability at a point which is strictly located on the left of $y_T$.

\begin{prop}[Bang-bang principle for safe trajectories]
If an optimal trajectory is safe for \eqref{eq:SOP} in the sense of \eqref{eq:safe}, then the optimal control is bang-bang. 
\end{prop}

\begin{proof}
Suppose by contradiction that we are given a safe optimal trajectory $(y^*,v^*)$ driven by a control $u^*$ that is not bang-bang. It can be easily verified that the data of \eqref{eq:SOP} satisfy Hypotheses \ref{hyp:MAC}, so that by Theorem \ref{Theo:uncontrolledDiff}, there exist stochastic processes $p_y^*,p_v^* \in C^2_{\mathcal{F}}([0,T] \times\Omega,\mathbb{R})$ and $q_y^*,q_v^* \in L^2_{\mathcal{F}}([0,T] \times\Omega,\mathbb{R})$ such that
\begin{equation}
\label{eq:AdjointSOP}
\left\{
\begin{aligned}
\mathrm{d}p_y^*(t) & = q_y^*(t) \mathrm{d}W_t, \hspace{1cm} p_y^*(T) = \xi^*(y_T-y^*(T)), \\
\mathrm{d}p_v^*(t) & = - p_y^*(t) \mathrm{d}t + q_v^*(t)  \mathrm{d}W_t, \hspace{1.37cm} p_v^*(T) =0,
\end{aligned}
\right.
\end{equation}
where in particular $\xi^* \in \partial (\textnormal{AV@R$_{\alpha}$})(0)$, and for which the maximization condition 
\begin{equation}
\label{eq:MaximizationSOP}
p_v^*(t) \, u^*(t) = \max_{u \in [-1,1]} \,  p_v^*(t) \,  u
\end{equation}
holds almost everywhere. Since we assumed that $u^*$ is not bang-bang, as a consequence of \eqref{eq:MaximizationSOP} there must exist a closed interval $I \subset [0,T]$ over which $p_v^* = 0$. Besides, it follows from standard properties of the Brownian motion applied to \eqref{eq:AdjointSOP} that 
\begin{equation}
\label{eq:ExpAdjointSOP}
\tfrac{\mathrm{d}}{\mathrm{d} t} \mathbb{E}[ p_y^*(t) ] = 0 \qquad \text{and} \qquad \tfrac{\mathrm{d}}{\mathrm{d} t} \mathbb{E}[ p_v^*(t) ] = - \mathbb{E}[ p_y^*(t) ]
\end{equation}
for all times $t \in [0,T]$. Since $t \in [0,T] \mapsto \mathbb{E}[p_v^*(t)]$ is Lipschitz by construction, it necessarily holds that $\mathbb{E}[ p_v^*(t) ] = 0$ on $I$, so that $\mathbb{E}[ p_y^*(t) ] = 0$ on $I$ as well, and thus
\begin{equation*}
\mathbb{E}[\xi^*(y^*(T)-y_T)] = 0
\end{equation*}
thanks to the uniqueness of solutions to \eqref{eq:ExpAdjointSOP}. Observing now that $\xi^* \in \partial \textnormal{AV@R$_{\alpha}$}(0)$ by construction, it follows from Definition \ref{def:Risk} and Theorem \ref{thm:RiskMeasures} that 
\begin{equation*}
\begin{aligned}
\mathbb{E}[\xi^*(y^*(T)-y_T)] & \leq \max_{\xi \in \partial (\textnormal{AV@R$_{\alpha}$})(0)} \mathbb{E}[\xi (y^*(T)-y_T)] \\
& = \textnormal{AV@R$_{\alpha}$}(y^*(T)-y_T) = \textnormal{AV@R$_{\alpha}$}(y^*(T)) - y_T.
\end{aligned}
\end{equation*}
In particular, we then recover that $y_T \leq \textnormal{AV@R$_{\alpha}$}(y^*(T))$ and  the optimal trajectory is not safe, which contradicts our primary assumption. 
\end{proof}

\section{Conclusion and perspectives} \label{sec:conclusion}

\setcounter{equation}{0} \renewcommand{\theequation}{\thesection.\arabic{equation}}

In this paper, we developed a new method for proving a first-order version of the Pontryagin Maximum Principle for non-smooth risk-averse optimal control problems, based on set-valued linearisations. The main incentive to do so was to produce optimality conditions that could encompass typical risk functions such as the AV@R, which is merely directionally differentiable. In the future, we aim at furthering these investigations in three main directions.

Firstly, we want to see whether it is feasible to weaken or remove the convexity assumptions on the dynamics. Owing to the lack of relaxation property for sollutions of \eqref{eq:SDI} illustrated in Remark \ref{rmk:Ito}, this will most likely call for innovative proof strategies. Secondly, we want to leverage the optimality conditions proposed here to design efficient numerical methods for solving risk-averse optimal control problems, such as indirect risk-averse shooting methods. Lastly, we plan to investigate whether the optimality conditions discussed in this article might yield other important structure properties on risk-averse optimal controls, such as \textit{semi-Markovianity}. Usually, the fact that optimal controls exhibit a Markovian dependance with respect to the state variable usually stems from the dynamic programming and HJB equations. While these latter are still largely unavailable in the risk-averse settings, we hope that our risk-averse PMP may take over and be sufficiently powerful to carry out the analysis.


{\footnotesize
\bibliographystyle{plain}
\bibliography{ref}
}

\end{document}